\documentclass[english]{article}
\usepackage[T1]{fontenc}
\usepackage[latin9]{inputenc}
\usepackage{refstyle}
\usepackage{enumitem}
\usepackage{amsmath}
\usepackage{amsthm}
\usepackage{amssymb}
\usepackage{setspace}
\makeatletter

\AtBeginDocument{\providecommand\secref[1]{\ref{sec:#1}}}
\AtBeginDocument{\providecommand\lemref[1]{\ref{lem:#1}}}
\AtBeginDocument{\providecommand\corref[1]{\ref{cor:#1}}}
\AtBeginDocument{\providecommand\subsecref[1]{\ref{subsec:#1}}}
\AtBeginDocument{\providecommand\thmref[1]{\ref{thm:#1}}}
\AtBeginDocument{\providecommand\exaref[1]{\ref{exa:#1}}}
\AtBeginDocument{\providecommand\propref[1]{\ref{prop:#1}}}
\RS@ifundefined{subsecref}
  {\newref{subsec}{name = \RSsectxt}}
  {}
\RS@ifundefined{thmref}
  {\def\RSthmtxt{theorem~}\newref{thm}{name = \RSthmtxt}}
  {}
\RS@ifundefined{lemref}
  {\def\RSlemtxt{lemma~}\newref{lem}{name = \RSlemtxt}}
  {}

\theoremstyle{plain}
\newtheorem{thm}{\protect\theoremname}[section]
  \theoremstyle{definition}
  \newtheorem{defn}[thm]{\protect\definitionname}
  \theoremstyle{plain}
  \newtheorem{lem}[thm]{\protect\lemmaname}
 \newlist{casenv}{enumerate}{4}
 \setlist[casenv]{leftmargin=*,align=left,widest={iiii}}
 \setlist[casenv,1]{label={{\itshape\ \casename} \arabic*.},ref=\arabic*}
 \setlist[casenv,2]{label={{\itshape\ \casename} \roman*.},ref=\roman*}
 \setlist[casenv,3]{label={{\itshape\ \casename\ \alph*.}},ref=\alph*}
 \setlist[casenv,4]{label={{\itshape\ \casename} \arabic*.},ref=\arabic*}
  \theoremstyle{plain}
  \newtheorem{cor}[thm]{\protect\corollaryname}
  \theoremstyle{definition}
  \newtheorem{example}[thm]{\protect\examplename}
  \theoremstyle{remark}
  \newtheorem{rem}[thm]{\protect\remarkname}
  \theoremstyle{plain}
  \newtheorem{prop}[thm]{\protect\propositionname}

\@ifundefined{date}{}{\date{}}
\AtBeginDocument{\providecommand\exaref[1]{\ref{exa:#1}}}
\AtBeginDocument{\providecommand\propref[1]{\ref{prop:#1}}}
\RS@ifundefined{propref}{\newref{prop}{name =Proposition~,names = Propositions~}}{}
\AtBeginDocument{\providecommand\corref[1]{\ref{cor:#1}}}
\RS@ifundefined{corref}{\newref{cor}{name =Corollary~,names = Corollaries~}}{}
\RS@ifundefined{exaref}{\newref{exa}{name =Example~,names = Examples~}}{}
\AtBeginDocument{\renewcommand\subsecref[1]{Section~\ref{subsec:#1}}}
\AtBeginDocument{\renewcommand\secref[1]{Section~\ref{sec:#1}}}

\def\@fnsymbol#1{\ensuremath{\ifcase#1\or *\or **\or \ddagger\or
   \mathsection\or \mathparagraph\or \|\or \dagger\dagger
   \or \ddagger\ddagger \else\@ctrerr\fi}}

\usepackage{tikz}

\usepackage{tocbibind}
\usepackage{youngtab}
\usepackage{authblk}
\usetikzlibrary{matrix}
\usepackage{bbding}
\usepackage{geometry}
\usepackage{ytableau}

\tikzset{   pt/.style={insert path={node[scale=2]{.}}},   dnup/.style={insert path={ [pt] .. controls +(0,1) and +(0,-1) .. +(#1,2) [pt]}},   dndn/.style={insert path={ [pt] .. controls +(0,0.25) and +(0,0.25) .. +(#1,0) [pt]}},   upup/.style={insert path={ [pt] .. controls +(0,-0.25) and +(0,-0.25) .. +(#1,0) [pt]}}, upup2/.style={insert path={ [pt] .. controls +(0,-0.5) and +(0,-0.5) .. +(#1,0) [pt]}}, }

\newcommand{\Rt}{\widetilde{\mathcal{R}}_E}
\newcommand{\Lt}{\widetilde{\mathcal{L}}_E}
\newcommand{\Ht}{\widetilde{\mathcal{H}}_E}

\DeclareMathOperator{\Hom}{Hom}

\DeclareMathOperator{\id}{id}

\DeclareMathOperator{\tr}{tr}

\DeclareMathOperator{\im}{\mathsf{im}}

\DeclareMathOperator{\PT}{\mathcal{PT}}
\DeclareMathOperator{\PR}{\mathcal{P}}

\DeclareMathOperator{\IS}{\mathcal{IS}}
\DeclareMathOperator{\T}{\mathcal{T}}
\DeclareMathOperator{\B}{\mathcal{B}}
\DeclareMathOperator{\Rc}{\mathcal{R}}
\DeclareMathOperator{\Lc}{\mathcal{L}}
\DeclareMathOperator{\Hc}{\mathcal{H}}
\DeclareMathOperator{\Dc}{\mathcal{D}}
\DeclareMathOperator{\Jc}{\mathcal{J}}

\DeclareMathOperator{\dom}{\mathsf{dom}}
\DeclareMathOperator{\Irr}{\mathsf{Irr}}

\DeclareMathOperator{\Ind}{Ind}
\DeclareMathOperator{\Coind}{Coind}

\DeclareMathOperator{\op}{op}

\DeclareMathOperator{\UT}{UT}
\DeclareMathOperator{\coker}{coker}

\newcommand{\LeqRt}{\leq_{\widetilde{\mathcal{R}}_E}}
\newcommand{\LeqLt}{\leq_{\widetilde{\mathcal{L}}_E}}
\newcommand{\LneqqLt}{\lneqq_{\widetilde{\mathcal{L}}_E}}

\def\RSlemtxt{Lemma~}
\def\RSthmtxt{Theorem~}

\usepackage{babel}
\providecommand{\corollaryname}{Corollary}
  \providecommand{\definitionname}{Definition}
  \providecommand{\examplename}{Example}
  \providecommand{\lemmaname}{Lemma}
  \providecommand{\propositionname}{Proposition}
  \providecommand{\remarkname}{Remark}
 \providecommand{\casename}{Case}
\providecommand{\theoremname}{Theorem}

\usepackage{babel}
\providecommand{\corollaryname}{Corollary}
  \providecommand{\definitionname}{Definition}
  \providecommand{\examplename}{Example}
  \providecommand{\lemmaname}{Lemma}
  \providecommand{\propositionname}{Proposition}
  \providecommand{\remarkname}{Remark}
 \providecommand{\casename}{Case}
\providecommand{\theoremname}{Theorem}

\makeatother

\usepackage{babel}
  \providecommand{\corollaryname}{Corollary}
  \providecommand{\definitionname}{Definition}
  \providecommand{\examplename}{Example}
  \providecommand{\lemmaname}{Lemma}
  \providecommand{\propositionname}{Proposition}
  \providecommand{\remarkname}{Remark}
 \providecommand{\casename}{Case}
\providecommand{\theoremname}{Theorem}

\begin{document}

\title{Ehresmann Semigroups Whose Categories are EI and Their Representation
Theory : Extended Version}

\author{Stuart Margolis \thanks{Department of Mathematics, Bar Ilan University, 52900 Ramat Gan, Israel }
\\
 \Envelope \, margolis@math.biu.ac.il\\
 \and Itamar Stein\thanks{Mathematics Unit, Shamoon College of Engineering, 77245 Ashdod, Israel }\\
 \Envelope \, Steinita@gmail.com}
\maketitle
\begin{abstract}
We study simple and projective modules of a certain class of Ehresmann
semigroups, a well-studied generalization of inverse semigroups. Let
$S$ be a finite right (left) restriction Ehresmann semigroup whose
corresponding Ehresmann category is an EI-category, that is, every
endomorphism is an isomorphism. We show that the collection of finite
right restriction Ehresmann semigroups whose categories are EI is
a pseudovariety. We prove that the simple modules of the semigroup
algebra $\Bbbk S$ (over any field $\Bbbk$) are formed by inducing
the simple modules of the maximal subgroups of $S$ via the corresponding
Sch\"{u}tzenberger module. Moreover, we show that over fields with good
characteristic the indecomposable projective modules can be described
in a similar way but using generalized Green's relations instead of
the standard ones. As a natural example, we consider the monoid $\PT_{n}$
of all partial functions on an $n$-element set. Over the field of
complex numbers, we give a natural description of its indecomposable
projective modules and obtain a formula for their dimension. Moreover,
we find certain zero entries in its Cartan matrix. 
\end{abstract}


\section{Introduction}

A semigroup $S$ is called \emph{inverse} if every element $a\in S$
has a unique inverse, that is, an element $b\in S$ such that $aba=a$
and $bab=b$. Inverse semigroups are one of the central objects of
study in semigroup theory (see \cite{Lawson1998}) and in particular
their representation theory is well-studied (see \cite[Part IV]{Steinberg2016}).
There are several generalizations of inverse semigroups that keep
some of their properties and structure. In this paper we discuss representations
of a generalization called Ehresmann semigroups. Let $E$ be a subsemilattice
of a semigroup $S$ (that is, a commutative subsemigroup of idempotents).
Define two equivalence relations $\Lt$ and $\Rt$ in the following
way: Elements $a,b\in S$ satisfy $a\Lt b$ ($a\Rt b$) if $a$ and
$b$ have the same set of right (respectively, left) identities from
$E$. We also define $\Ht=\Rt\cap\Lt$. Assume that every $\Lt$ and
$\Rt$ class contains precisely one idempotent from $E$ denoted $a^{\ast}$
and $a^{+}$ respectively. The semigroup $S$ is called \emph{Ehresmann}
(or \emph{$E$-Ehresmann }if the set $E$ is to be emphasized\emph{)}
if $\Lt$ is a right congruence and $\Rt$ is a left congruence. An
Ehresmann semigroup is called right restriction if the identity $ea=a(ea)^{\ast}$
holds for every $e\in E$ and $a\in S$. The main class of semigroups
we consider in this paper are finite right restriction Ehresmann semigroups
with the additional property that the $\Ht$-class of every $e\in E$
is a group. We call such semigroups right restriction EI-Ehresmann
because the corresponding Ehresmann category is an EI-category (that
is, every endomorphism is an isomorphism). This is a generalization
of a class known in the literature as (finite) weakly ample semigroups
(see \cite[Section 4]{Gould2010}). Natural examples of such semigroups
are the monoid $\PT_{n}$ of all partial functions on the set $\{1,\ldots,n\}$
and the monoid of full domain partitions on a finite set introduced
in \cite{East2020}. In \secref{Classes_of_EU_Ehresmann} we prove
that this class is a pseudovariety in signature $(2,1,1)$. We also
discuss cases when such semigroups are embeddable in $\PT_{n}$ as
a bi-unary semigroup.

In \secref{Simple_modules} we show that the simple modules of the
semigroup algebra $\Bbbk S$ (over any field $\Bbbk$) of a finite
right restriction EI-Ehresmann semigroup are given by induced Sch\"{u}tzenberger
modules of the simple modules of the maximal subgroups of $S$. Unlike
inverse semigroups, Ehresmann semigroups algebras need not be semisimple,
even over fields of ``good'' characteristic. Therefore, an Ehresmann
semigroup can have non-semisimple projective modules. Another goal
of this paper is to describe the indecomposable projective modules
of finite right restriction EI-Ehresmann semigroups. For this we need
a construction that replaces Green's $\Lc$-classes by $\Lt$-classes
as the basis of induction. Let $S$ be a finite semigroup and let
$E\subseteq S$ be a subset of idempotents. Choose an $e\in E$ and
let $\Lt(e)$ be its $\Lt$-class. In \secref{Construction_of_Lt_modules}
we characterize certain cases where $G_{e}$ (the maximal subgroup
with unit element $e$) acts on the right of $\Lt(e)$. In particular
we prove that $G_{e}$ acts on the right of $\Lt(e)$ for every Ehresmann
semigroup. If we denote by $\Bbbk\Lt(e)$ the $\Bbbk$-vector space
with basis $\Lt(e)$, this implies that $\Bbbk\Lt(e){\displaystyle \bigotimes_{\Bbbk G_{e}}V}$
is a $\Bbbk S$-module for every $\Bbbk G_{e}$-module $V$.

A very successful way to study algebras of inverse semigroups or their
generalizations is to relate them to algebras of an associated category
(or a ``partial semigroup''). This is often done with an appropriate
M�bius function as in \cite{Guo2008,Guo2012,Ji2016,Solomon1967,Stein2016,Steinberg2006,Steinberg2008,Wang2017}.
In particular, the second author has proved in \cite{Stein2017,Stein2018erratum}
the following theorem. Let $S$ be a finite right restriction Ehresmann
semigroup. Then the semigroup algebra $\Bbbk S$ (over any field $\Bbbk$)
is isomorphic to the category algebra $\Bbbk\mathcal{C}$ for the
associated Ehresmann category $\mathcal{C}$. This result has led
to several applications in the representation theory of monoids of
partial functions \cite{Stein2019,Stein2020}. Now assume again that
$S$ is a finite right restriction EI-Ehresmann semigroup and also
assume that the order of every subgroup of $S$ is invertible in $\Bbbk$.
In \secref{ProjectiveModules} we use the isomorphism between the
semigroup algebra and the corresponding category algebra mentioned
above to prove that any module of the form $\Bbbk\Lt(e){\displaystyle \bigotimes_{\Bbbk G_{e}}V}$
(where $V$ is a simple $\Bbbk G_{e}$-module) is an indecomposable
projective module of $\Bbbk S$. Moreover, any indecomposable projective
module is of this form. This gives a purely semigroup theoretic construction
and it is very similar to other known constructions in the representation
theory of semigroups \cite{Ganyushkin2009a,Margolis2018B}. We also
give a formula for the dimension of this module over an algebraically
closed field. As an application, we consider the monoid $\PT_{n}$
in the case $\Bbbk=\mathbb{C}$. It was already proved that the simple
and indecomposable injective modules of $\mathbb{C}\PT_{n}$ can be
described using induced and co-induced representations respectively
(essentially this is a consequence of \cite[Theorem 4.4]{Margolis2011}).
In \secref{PartialFunctions} we complete this picture by giving a
description of the indecomposable projective modules of $\mathbb{C}\PT_{n}$
along with the natural epimorphism onto their simple images. The general
formula for the dimension of the indecomposable projective modules
boils down in this case to a combinatorial formula that sums up certain
Kostka numbers. We also draw some conclusions regarding the Cartan
matrix of $\mathbb{C}\PT_{n}$ by showing that certain entries of
it are zero. Finally, we give a counter-example to the above isomorphism
between the algebras of an Ehresmann semigroup and the associated
category in the case where $S$ is not right (or left) restriction
hence showing that this requirement cannot be omitted. We also give
an example of an EI-Ehresmann semigroup that is neither left nor right
restriction but its algebra is isomorphic to the algebra of the corresponding
Ehresmann category. 

\section{Preliminaries}

\subsection{Semigroups}

Let $S$ be a semigroup and let $S^{1}=S\cup\{1\}$ be the monoid
formed by adjoining a formal unit element. We denote by $\Hc$, $\Lc$,
$\Rc$ and $\Jc$ the usual Green's equivalence relations: 
\begin{align*}
a\Rc b\iff & aS^{1}=bS^{1}\\
a\Lc b\iff & S^{1}a=S^{1}b\\
a\Jc b\iff & S^{1}aS^{1}=S^{1}bS^{1}
\end{align*}
and $\Hc=\Rc\cap\Lc$. In a finite semigroup, it is known that $\Jc=\Rc\circ\Lc=\Lc\circ\Rc$,
which is Green's relation $\Dc$. We denote by $E(S)$ the set of
all idempotents of $S$. It is well-known that if $e\in E(S)$ then
its $\Hc$-class forms a group that we will denote by $G_{e}$. This
is the maximal subgroup of $S$ with unit element $e$. An element
$a\in S$ is called\emph{ regular} if there exists $b\in S$ such
that $aba=a$. Now assume that $S$ is finite. A $\Jc$-class $J$
is called \emph{regular} if it contains a regular element. It is known
that in this case all its elements are regular and every $\Rc$ and
every $\Lc$ class contains an idempotent. Moreover, if $J$ is a
regular $\Jc$-class then for every two idempotents $e_{1},e_{2}\in J$
we have $G_{e_{1}}\simeq G_{e_{2}}$. So we can denote this group
by $G_{J}$ - the unique maximal group associated with $J$. Let $J$
be a regular $\Jc$-class of $S$. The corresponding regular factor
is a $0$-simple finite semigroup and thus by Rees Theorem is isomorphic
to a regular Rees matrix semigroup $M^{0}(G_{e},A,B,P)$. See \cite[Chapter 3]{Howie1995}
and the Appendix of \cite{qTheory2009} for details. We denote by
$\B_{n}$ the monoid of all binary relation on the set $\{1,\ldots,n\}$
with composition as operation and by $\PT_{n}$ its submonoid of all
partial functions. The submonoid of all total functions is denoted
$\T_{n}$ and $\IS_{n}$ will be our notation for the submonoid consisting
of all injective partial functions. We remark that we compose binary
relations (and in particular partial functions) from right to left
in this paper. A standard textbook for elementary semigroup theory
is \cite{Howie1995}. For a text devoted to finite semigroups see
\cite{qTheory2009}.

\subsection{Ehresmann semigroups and categories\label{subsec:EhresmannSgp}}

Let $S$ be a semigroup and let $E\subseteq S$ be a subset of idempotents.
We define two preorders $\LeqLt$ and $\LeqRt$ on $S$ by 
\[
a\LeqLt b\iff(\forall e\in E\quad be=b\Rightarrow ae=a)
\]
\[
a\LeqRt b\iff(\forall e\in E\quad eb=b\Rightarrow ea=a).
\]
The corresponding equivalence relations are denoted $\Lt$ and $\Rt$.
\[
a\Lt b\iff(\forall e\in E\quad be=b\Leftrightarrow ae=a)
\]
\[
a\Rt b\iff(\forall e\in E\quad eb=b\Leftrightarrow ea=a).
\]

Moreover, we define $\Ht=\Lt\cap\Rt$. We remark that $\Lc\subseteq\Lt$,
$\Rc\subseteq\Rt$ and $\Hc\subseteq\Ht$. A subset $E\subseteq S$
of idempotents is called a \emph{subsemilattice} if it is a commutative
subsemigroup. It is well known that any commutative semigroup of idempotents
has the structure of a semilattice (i.e. a poset where every two elements
have a meet) if one defines $a\leq b$ whenever $ab=ba=a$. A semigroup
$S$ with a subsemilattice $E\subseteq S$ is called \emph{right $E$-Ehresmann
}if every $\Lt$-class contains a unique idempotent from $E$ and
$\Lt$ is a right congruence. We denote the unique idempotent in the
$\Lt$-class of $a$ by $a^{\ast}$. Note that $a^{\ast}$ is the
unique minimal element $e\in E$ such that $ae=a$. It is well known
that $\Lt$ is a right congruence if and only if the identity $(ab)^{\ast}=(a^{\ast}b)^{\ast}$
holds for every $a,b\in S$.

Dually, we can consider semigroups for which every $\Rt$ class contains
a unique idempotent. We denote the unique idempotent in the $\Rt$
class of $a$ by $a^{+}$. Such semigroup is called left $E$-Ehresmann
if $\Rt$ is a left congruence, or equivalently if $(ab)^{+}=(ab^{+})^{+}$
for every $a,b\in S$. A semigroup $S$ with a subsemilattice $E\subseteq S$
is called \emph{$E$-Ehresmann }if it is both left and right $E$-Ehresmann.
The semilattice $E$ is also called the set of \emph{projections }of
$S$\emph{.} If $S$ is an $E$-Ehresmann semigroup we can define
two partial orders by $a\leq_{l}b$ ($a\leq_{r}b$) if and only if
$a=ba^{\ast}$ (respectively, $a=a^{+}b$) which is also equivalent
to $a=be$ for some $e\in E$ (respectively, $a=eb$ for some $e\in E$).

It is known that $E$-Ehresmann semigroups form a variety of $(2,1,1)$-algebras
(where $^{+}$ and $^{\ast}$ are the unary operations). See \cite{Gould2010b}
for a list of the identities. We use them without reference in this
paper. Therefore, we can drop the explicit mention of the set $E$
and consider an Ehresmann semigroup $S$ to be a bi-unary semigroup
where the set of projections $E$ is determined implicitly by the
unary operations: 
\[
E=\{a^{\ast}\mid a\in S\}=\{a^{+}\mid a\in S\}
\]

A right (left) Ehresmann semigroup $S$ is called \emph{right (respectively,
left) restriction} if the ``right ample'' (respectively, ``left
ample'') identity $ea=a(ea)^{\ast}$ (respectively, $ae=(ae)^{+}a$)
holds for every $a\in S$ and $e\in E$.

For every Ehresmann semigroup $S$, we can associate a category ${\bf C}(S)=\mathcal{C}$
in the following way. The object set of ${\bf C}(S)$ is the set $E$
of projections and the morphisms of ${\bf C}(S)$ are in one-to-one
correspondence with elements of $S$. For every $a\in S$, we associate
a morphism $C(a)\in\mathcal{C}^{1}$ with domain $a^{+}$ and range
$a^{\ast}$. If the range of $C(a)$ is the domain of $C(b)$ the
composition $C(a)\cdot C(b)$ is defined to be $C(ab)$. The convention
is to compose morphisms in Ehresmann categories ``from left to right''.
However, for us it will be more convenient to use composition ``from
right to left''. In this case, we will associate to every $a\in S$
a morphism $C(a)\in\mathcal{C}^{1}$ with domain $a^{\ast}$ and range
$a^{+}$ and if the range of $C(a)$ is the domain of $C(b)$ the
composition $C(b)\cdot C(a)$ is defined to be $C(ba)$. For more
facts and proofs on Ehresmann semigroups and Ehresmann categories
the reader is referred to \cite{Gould2010,Gould2010b}.

\subsection{Algebras and representations}

Let $A$ be a $\mathbb{\Bbbk}$-algebra where $\Bbbk$ is a field.
Unless stated otherwise, we assume that algebras are unital and finite
dimensional. Likewise, when we say that $M$ is a module over $A$
(or an $A$-module) we mean that $M$ is a finite dimensional left
module over $A$. We will denote the set of \emph{simple} or \emph{irreducible}
modules of $A$ by $\Irr A$. Basic facts about the representation
theory of finite dimensional algebras are found in \cite{Assem2006},
which we will use without further mention. We will be interested in
semigroup algebras and category algebras. The semigroup algebra $\mathbb{\Bbbk}S$
of a finite semigroup $S$ is defined in the following way. It is
a vector space over $\Bbbk$ with basis the elements of $S$, that
is, it consists of all formal linear combinations: 
\[
\{k_{1}s_{1}+\ldots+k_{n}s_{n}\mid k_{i}\in\mathbb{\Bbbk},\,s_{i}\in S\}.
\]
The multiplication in $\mathbb{\Bbbk}S$ is the linear extension of
the semigroup multiplication. The category algebra $\mathbb{\Bbbk}\mathcal{C}$
of a finite category $\mathcal{C}$ is defined in the following way.
It is a vector space over $\mathbb{\Bbbk}$ with basis the morphisms
of $\mathcal{C}$, that is, it consists of all formal linear combinations:
\[
\{k_{1}m_{1}+\ldots+k_{n}m_{n}\mid k_{i}\in\mathbb{\Bbbk},\,m_{i}\in\mathcal{C}^{1}\}.
\]
The multiplication in $\mathbb{\Bbbk}\mathcal{C}$ is the linear extension
of the following: 
\[
m^{\prime}\cdot m=\begin{cases}
m^{\prime}m & \text{if \ensuremath{m^{\prime}\cdot m} is defined}\\
0 & \text{otherwise}.
\end{cases}
\]

\subsection{Representations of finite semigroups and groups}

The main reference for the representation theory of finite monoids
is \cite{Steinberg2016}. Let $S$ be a finite semigroup, let $J$
be regular $\Jc$-class and fix an idempotent $e\in J$. Let $\Bbbk$
be a field. The semigroup $S$ acts by partial functions on $\Lc(e)$,
the $\Lc$-class of $e$, according to 
\[
s\cdot x=\begin{cases}
sx & sx\in\Lc(e)\\
\text{undefined} & \text{otherwise}
\end{cases}
\]
for $s\in S$ and $x\in\Lc(e)$. Moreover, the group $G_{e}\simeq G_{J}$
acts on $\Lc(e)$ by right multiplication. Therefore, $\Bbbk\Lc(e)$
is a $\Bbbk S-\Bbbk G_{J}$ bi-module (where $\Bbbk\Lc(e)$ is the
set of all formal linear combinations of elements of $\Lc(e)$). Dually,
$\Bbbk\Rc(e)$ is a $\Bbbk G_{J}-\Bbbk S$ bi-module. Let $V$ be
a $\Bbbk G_{J}$-module, we define the induced module corresponding
to $J$ and $V$ as follows:

\[
\Ind_{G_{J}}(V)=\Bbbk\Lc(e)\bigotimes_{\Bbbk G_{J}}V
\]

It can be shown that this module does not depend on the specific choice
of $e\in J$. This justifies the notation $\Ind_{G_{J}}(V)$. If $V\in\Irr\Bbbk G_{J}$
then the maximal semisimple image of $\Ind_{G_{J}}(V)$ is a simple
$\Bbbk S$-module. Moreover, the Clifford-Munn-Ponizovskii theorem
states that this induces a one-to-one correspondence between simple
modules of $\Bbbk S$ and pairs $(J,V)$ where $J$ is regular $\Jc$
class and $V\in\Irr G_{J}$. Therefore, semisimple images of induced
modules of the above form give a complete list of all simple modules
of $\Bbbk S$ up to isomorphism. More details on Clifford-Munn-Ponizovskii
theory can be found in \cite{Ganyushkin2009a} and \cite[Chapter 5]{Steinberg2016}.

We assume familiarity with basic results on the representation theory
of finite groups \cite{Curtis1966}. For the representation theory
of the symmetric group and its connection to Young Tableaux, we refer
to \cite{James1981,Sagan2001}. We use these classical results without
further mention.

We say that a Young tableau is semi-standard if its columns are increasing
and its rows are non-decreasing. The Kostka number $K_{\lambda\mu}$
is the number of semistandard Young tableaux with shape $\lambda$
and content $\mu$ (see \cite[Section 2.11]{Sagan2001}).

For every composition $\lambda=[\lambda_{1},\ldots,\lambda_{r}]\vDash k$
we can associate the subgroup $S_{\lambda}=S_{\lambda_{1}}\times\cdots\times S_{\lambda_{r}}$
of $S_{k}$, called the \emph{Young subgroup} corresponding to\emph{
$\lambda$}. Let $\mu=[\mu_{1},\ldots,\mu_{r}]\vDash k$ and denote
by $\tr_{\mu}$ the trivial module of the algebra of the Young subgroup
$S_{\mu}$. The module $\Ind_{S_{\mu}}^{S_{k}}\tr_{\mu}$ is called
the \emph{Young module} corresponding to $\mu$. For every partition
$\lambda\vdash k$ the multiplicity of the Specht module $S^{\lambda}$
in the decomposition of $\Ind_{S_{\mu}}^{S_{k}}\tr_{\mu}$ is precisely
$K_{\lambda\mu}$ - the Kostka number of $\lambda,\mu$.

\section{\label{sec:Classes_of_EU_Ehresmann}Some classes of Ehresmann semigroups}

By an endomorphism of a category $\mathcal{C}$ we mean a morphism
$f$ whose domain and range are the same object. For an object $c$
of $\mathcal{C}$ we call $\mathcal{C}(c,c)$, the endomorphism monoid
at $c$. 
\begin{defn}
A category is called an \emph{EI-category} if every endomorphism is
an isomorphism, that is, the endomorphism monoids are groups. 
\end{defn}
There is a vast literature about representations of EI-categories
and their applications (for some examples see \cite{Li2011,Luck1989,Dieck1987,Webb2007}). 
\begin{defn}
Let $S$ be an Ehresmann semigroup. We call $S$ an \emph{EI-Ehresmann}
semigroup if the associated Ehresmann category ${\bf C}(S)$ is an
EI-category. 
\end{defn}
Note that $a\in\Ht(e)$ if and only if $C(a)$ is an endomorphism
of $e$ in the associated Ehresmann category ${\bf C}(S)$. Therefore
$S$ is EI-Ehresmann if and only if $\Ht(e)$ is a group for every
projection $e\in E$.

\subsection{EI-Ehresmann semigroups}

We start with some characterizations of EI-Ehresmann semigroups. Recall
that the \emph{sandwich set} of two idempotents $f_{1},f_{2}\in S$
is defined by 
\[
S(f_{1},f_{2})=\{h\in E(S)\mid f_{2}hf_{1}=h,\quad f_{1}hf_{2}=f_{1}f_{2}\}.
\]

Let $V(s)$ be the set of inverses of an element $s\in S$. It is
well known that the sandwich set is non-empty if and only if $f_{1}f_{2}$
is a regular element of $S$. Furthermore, 

\[
S(f_{1},f_{2})=\{h\in V(f_{1}f_{2})\cap E(S)\mid f_{2}hf_{1}=h\}.
\]

In particular, if $f_{1}$ and $f_{2}$ commute, then $S(f_{1},f_{2})$
is non-empty since in this case, $f_{1}f_{2}$ is an idempotent. See
\cite[Chapter 2.5]{Howie1995} for details. 
\begin{lem}
\label{lem:EU_Ehresmann_characterization}Let $S$ be a finite Ehresmann
semigroup, The following are equivalent. 
\begin{enumerate}
\item $S$ is an EI-Ehresmann semigroup. 
\item Every idempotent $f\in E(S)$ satisfies $f^{\ast}ff^{+}=f^{\ast}f^{+}$. 
\item The projection $f^{\ast}f^{+}$ is an inverse of $f$ for every idempotent
$f\in E(S)$. 
\item Every idempotent $f\in E(S)$ satisfies $f\in S(f^{\ast},f^{+}).$ 
\end{enumerate}
\end{lem}
\begin{proof}
\begin{casenv}
\item[$(1\implies2$)] Observe that $f^{\ast}ff^{+}$ is an idempotent since 
\[
f^{\ast}ff^{+}\cdot f^{\ast}ff^{+}=f^{\ast}ff^{\ast}f^{+}ff^{+}=f^{\ast}fff^{+}=f^{\ast}ff^{+}.
\]
Moreover, 
\[
\left(f^{\ast}ff^{+}\right)^{+}=\left(f^{\ast}\left(ff^{+}\right)^{+}\right)^{+}=f^{\ast}\left(ff^{+}\right)^{+}=f^{\ast}\left(ff\right)^{+}=f^{\ast}f^{+}
\]
where the first and third equalities follow from the left congruence
identity $(ab)^{+}=(ab^{+})^{+}$. Likewise 
\[
\left(f^{\ast}ff^{+}\right)^{\ast}=\left(\left(f^{\ast}f\right)^{\ast}f^{+}\right)^{\ast}=\left(f^{\ast}f\right)^{\ast}f^{+}=\left(ff\right)^{\ast}f^{+}=f^{\ast}f^{+}
\]
where the first and third equalities follow from the right congruence
identity $(ab)^{\ast}=(a^{\ast}b)^{\ast}$. Therefore $f^{\ast}ff^{+}$
is $\Ht$-equivalent to $f^{\ast}f^{+}$. Since $\Ht(f^{\ast}f^{+})$
has a unique idempotent we obtain that $f^{\ast}ff^{+}=f^{\ast}f^{+}$. 
\item[$(2\implies1$)] Let $e\in E$ and let $f\in E(S)$ be an idempotent such that $f\Ht e$
so $f^{+}=f^{\ast}=e$. The assumption $f^{\ast}ff^{+}=f^{\ast}f^{+}$
now reduces to 
\[
efe=ee=e
\]
and clearly $efe=f^{+}ff^{\ast}=f$ so $f=e$ as required. 
\item[$(2\iff3\iff4$)] Immediate from the properties of sandwich sets noted before this
Lemma. 
\end{casenv}
\end{proof}
We can interpret the last result in the language of universal algebra,
namely (pseudo)varieties and quasivarieties. We refer to \cite{Almeida1994,universalalgebra}
for the basics of varieties, quasivarieties and pseudovarieties and
other concepts from Universal Algebra. If $x$ is an element of a
free profinite semigroup, then $x^{\omega}$ is the unique idempotent
in the closed subsemigroup generated by $x$. In particular, if $S$
is a finite semigroup, then $x^{\omega}$ is the unique idempotent
in the subsemigroup generated by $x$. 

The equivalence $1\iff2$ of \lemref{EU_Ehresmann_characterization}
immediately implies the following corollary in the language of pseudovarieties. 
\begin{cor}
\label{cor:Pseudo_variety_EU_Ehresmann}The class of all finite EI-Ehresmann
semigroups is a pseudovariety of finite bi-unary semigroups defined
by the identities of Ehresmann semigroups (see \cite{Gould2010b})
and the profinite identity 
\[
\left(x^{\omega}\right)^{\ast}x^{\omega}\left(x^{\omega}\right)^{+}=\left(x^{\omega}\right)^{\ast}\left(x^{\omega}\right)^{+}.
\]
\end{cor}
\begin{proof}
Let $S$ is a finite Ehresmann semigroup. If $S$ is EI-Ehresmann
then the required profinite identity holds by \lemref{EU_Ehresmann_characterization}
because $x^{\omega}$ is an idempotent for every $x\in S$ . In the
other direction, take $f\in E(S)$. Then $f^{\omega}=f$ so the assumption
yields 
\[
f^{\ast}ff^{+}=f^{\ast}f^{+}
\]
which implies that $S$ is EI-Ehresmann by \lemref{EU_Ehresmann_characterization}. 
\end{proof}
Let $V$ be a pseudovariety of finite groups and let $\overline{V}$
be the pseudovariety of finite semigroups whose subgroups belong to
$V$. Note that if $S$ is a finite EI-Ehresmann semigroups then every
$\Ht(e)$-class is a group and in particular, $\Ht(e)=\Hc(e)$ for
every $e\in E$. Therefore the pseudovariety of all finite Ehresmann
semigroups whose $\Ht$-classes (or equivalently, the endomorphism
groups in the category ${\bf C}(S)$) belongs to $V$ is precisely
the intersection between $\overline{V}$ and the pseudovariety of
\corref{Pseudo_variety_EU_Ehresmann}. 
\begin{example}
Consider the pseudovariety of all Ehresmann semigroups such that the
corresponding Ehresmann category ${\bf C}(S)$ is locally trivial
(i.e, the only endomorphisms are the identity functions). Some natural
monoids which belong to this pseudovariety are discussed in \cite{Stein2020}.
In this case the $\Ht$-classes are trivial so the subgroups of $S$
are trivial. Such semigroups are called aperiodic (or $\Hc$-trivial).
This pseudovariety can be described by the Ehresmann identities, and
the profinite identities 
\[
\left(x^{\omega}\right)^{\ast}x^{\omega}\left(x^{\omega}\right)^{+}=\left(x^{\omega}\right)^{\ast}\left(x^{\omega}\right)^{+},\quad x^{\omega}x=x^{\omega}
\]
where the second one is for the aperiodicity of $S$. 
\end{example}

\subsection{\label{subsec:Right_restriction_EU_Ehresmann}Right restriction EI-Ehresmann
semigroups}

Now we turn to the main class of semigroups that we want to discuss
in this paper: Semigroups that are both EI-Ehresmann and right restriction.
We start with the following lemma. 
\begin{lem}
\label{lem:Right_restriction_Ehresmann_leq}Let $S$ be a right restriction
Ehresmann semigroup. Then $f^{+}\leq f^{\ast}$ for every idempotent
$f\in E(S)$. 
\end{lem}
\begin{proof}
\footnote{We thank Professor Victoria Gould for this proof.}Let $f\in E(S)$
be an idempotent. According to the right congruence identity 
\[
f^{\ast}f^{+}=(f^{\ast}f^{+})^{+}=(f^{\ast}f)^{+}
\]
and according to the right ample identity 
\[
(f^{\ast}f)^{+}=\left(f(f^{\ast}f)^{\ast}\right)^{+}.
\]
Now, by the left congruence identity 
\[
\left(f(f^{\ast}f)^{\ast}\right)^{+}.=\left(f(ff)^{\ast}\right)^{+}=(ff^{\ast})^{+}=f^{+}
\]
so $f^{\ast}f^{+}=f^{+}$ and thus $f^{+}\leq f^{\ast}$ as required. 
\end{proof}
\begin{lem}
\label{lem:right_restriction_EU_Ehresmann_characterization}Let $S$
be a finite right restriction Ehresmann semigroup. The following are
equivalent. 
\begin{enumerate}
\item $S$ is an EI-Ehresmann semigroup. 
\item Every idempotent $f\in E(S)$ satisfies $ff^{+}=f^{+}$. 
\item Every regular element $a\in S$ satisfies $a\Rc a^{+}$. 
\end{enumerate}
\end{lem}
\begin{proof}
\begin{casenv}
\item[($1\implies2$)] If $S$ is EI-Ehresmann then $f^{\ast}ff^{+}=f^{\ast}f^{+}$ by \lemref{EU_Ehresmann_characterization}.
\lemref{Right_restriction_Ehresmann_leq} implies that 
\[
f^{\ast}f=f^{\ast}f^{+}f=f^{+}f=f
\]
so we establish $ff^{+}=f^{+}$ as required. 
\item[($2\implies3$)] If $a$ is regular then $a\Rc f$ for some idempotent $f\in E(S)$
and clearly $f^{+}=a^{+}$. Now $ff^{+}=f^{+}$ implies $f\Rc f^{+}$
hence $a\Rc a^{+}$ as required. 
\item[($3\implies1$)] Every idempotent $f\in E(S)$ is regular, so $fRf^{+}$ which says
that $fx=f^{+}$ for some $x\in S$. Now 
\[
ff^{+}=ffx=fx=f^{+}
\]
so $f^{\ast}ff^{+}=f^{\ast}f^{+}$ hence $S$ is EI-Ehresmann by \lemref{EU_Ehresmann_characterization}. 
\end{casenv}
\end{proof}
\begin{rem}
A sandwich set $S(f_{1},f_{2})$, if not empty, is a rectangular band
(\cite[Proposition 2.5.3]{Howie1995}), i.e., a subsemigroup satisfying
$ghg=g$ for every $g,h\in S(f_{1},f_{2})$. The property $ff^{+}=f^{+}$
implies that $S(f^{\ast},f^{+})$ is a right zero semigroup (a rectangular
band with one $\Rc$-class). Indeed if $h\in S(f^{\ast},f^{+})$ then
\[
fh=ff^{+}hf^{\ast}=f^{+}hf^{\ast}=h
\]
and hence 
\[
f=fhf=hf
\]
so every element of $S(f^{\ast},f^{+})$ is $\Rc$-equivalent to $f$. 
\end{rem}
\begin{example}
\label{exa:EI_Ehresmann_not_Right_left_restriction}The following
example, taken from \cite[Example 5.12]{Stein2017}, shows that not
every EI-Ehresmann semigroup is right or left restriction, so right
restriction EI-Ehresmann semigroups form a proper subclass of EI-Ehresmann
semigroups.

Recall that $\T_{n}$ is the monoid of all total functions on an n-element
set. Denote by $\id$ the identity function and by ${\bf k}$ the
constant functions that sends every element to $k$. Define $S$ to
be the subsemigroup of $\T_{2}^{\op}\times\T_{2}$ containing the
six elements 
\[
({\bf 1,1}),({\bf 2},{\bf 1}),({\bf 1},{\bf 2}),({\bf 2},{\bf 2}),({\bf 1},\id),(\id,{\bf 1})
\]

It can be checked that $S$ is an $E$-Ehresmann semigroup with 
\[
E=\{({\bf 1,1}),({\bf 1},\id),(\id,{\bf 1})\}
\]
as its set of projections. The corresponding Ehresmann category $\mathcal{C}$
is given by the following drawing (recall that composition is ``right
to left''):
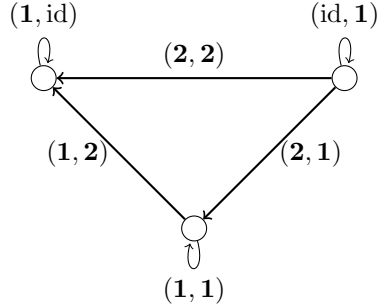
\begin{center}
\begin{tikzpicture}\path (0,2) node [shape=circle,draw] (e) {} edge [loop above] node {$ (\bf{1},\id)$} (); \path (4,2) node [shape=circle,draw] (f) {} edge [loop above] node {$(\id,\bf{1})$} ();\path (2,0) node [shape=circle,draw] (a) {} edge [loop below] node {$(\bf{1},\bf{1})$} (); \draw[thick,<-] (e)--(f) node [above,midway] {$ (\bf{2},\bf{2})$}; \draw[thick,<-] (e)--(a) node [left,midway] {$ (\bf{1},\bf{2})$}; \draw[thick,<-] (a)--(f) node [right,midway] {$ (\bf{2},\bf{1})$}; \end{tikzpicture} 
\par\end{center}
Clearly, $\mathcal{C}$ is an EI-category (in fact, it is even locally
trivial) so $S$ is an EI-Ehresmann semigroup. However, it is easy
to see that $({\bf 2},{\bf 2})^{+}=({\bf 1},\id)$ and $({\bf 2},{\bf 2})^{\ast}=(\id,{\bf 1})$
but $({\bf 2},{\bf 2})$ is not $\Lc$-related to $({\bf 1},\id)$
and not $\Rc$-related to $(\id,{\bf 1})$ so $S$ is neither left
nor right restriction. Note also that the sandwich set 
\[
S((\id,{\bf 1}),({\bf 1},\id))=\{({\bf 1,1}),({\bf 2},{\bf 1}),({\bf 1},{\bf 2}),({\bf 2},{\bf 2})\}
\]
is a rectangular band but not a right\textbackslash{}left zero semigroup. 
\end{example}
As in the case of general EI-Ehresmann semigroups, we have an immediate
corollary of \lemref{right_restriction_EU_Ehresmann_characterization}. 
\begin{cor}
The class of all finite right restriction EI-Ehresmann semigroups
is a pseudovariety of finite bi-unary semigroups defined by the identities
of Ehresmann semigroups, the right ample identity $b^{\ast}a=a(b^{\ast}a)^{\ast}$
and the profinite identity 
\[
x^{\omega}\left(x^{\omega}\right)^{+}=\left(x^{\omega}\right)^{+}.
\]
\end{cor}
The most natural example of a right restriction EI-Ehresmann semigroup
is the monoid $\PT_{n}$ of all partial functions on the set $\{1,\ldots,n\}$.
In \secref{PartialFunctions} we will consider at length its representation
theory. Let $A\subseteq\{1,\ldots,n\}$ and denote by $\id_{A}$ the
partial identity function on $A$. It is clear that $E=\{\id_{A}\mid A\subseteq\{1,\ldots,n\}\}$
is a subsemilattice of $\PT_{n}$ and it is well known that $\PT_{n}$
is a right restriction $E$-Ehresmann semigroup. Let $t:A\to B$ be
a partial function. We denote by $\dom(t)\subseteq A$ and $\im(t)\subseteq B$
the domain and image of $t$ and it is clear that $t^{+}=1_{\im(t)}$
and $t^{\ast}=1_{\dom(t)}$. It is also clear that an idempotent $f\in E(\PT_{n})$
must satisfy $\left.f\right|_{\im(f)}=1_{\im(f)}$ so the only idempotent
$f\in\PT_{n}$ with $\dom(f)=\im(f)$ is $f=1_{\dom(f)}$. Therefore,
$\PT_{n}$ is indeed an EI-Ehresmann semigroup. Since finite right
restriction EI-Ehresmann semigroups form a pseudovariety, it is clear
that any $(2,1,1)$-subalgebra of $\PT_{n}$ is also a right restriction
EI-Ehresmann semigroup. We now show that at least for regular semigroups
the converse also holds. 
\begin{prop}
Let $S$ be a regular finite right restriction EI-Ehresmann semigroup.
There is an embedding of bi-unary semigroups $\Phi:S\to\PT_{n}$ for
$n=|S|$. 
\end{prop}
\begin{proof}
The ``Cayley theorem'' for right restriction semigroups (see \cite[Theorem 6.2]{Gould2010})
says that there is a semigroup monomorphism $\Phi:S\to\PT_{n}$ such
that $\Phi(a^{\ast})=\Phi(a)^{\ast}$ for every $a\in S$. It remains
to show that $\Phi(a^{+})=\Phi(a)^{+}$ as well. First note that 
\[
\Phi(a^{+})=\Phi\left(\left(a^{+}\right)^{\ast}\right)=\Phi\left(a^{+}\right)^{\ast}
\]
so $\Phi(a^{+})$ is a partial identity of $\PT_{n}$. \lemref{right_restriction_EU_Ehresmann_characterization}
and the fact that $S$ is regular imply that $a\Rc a^{+}$ for every
$a\in S$. Therefore $\Phi(a)\Rc\Phi(a^{+})$ in $\PT_{n}$. Since
$\Phi(a)^{+}$ is the only partial identity which is $\Rc$ equivalent
to $\Phi(a)$ in $\PT_{n}$ we must conclude that $\Phi(a^{+})=\Phi(a)^{+}$
as required. 
\end{proof}
On the other hand, not every right restriction EI-Ehresmann semigroup
is embeddable in $\PT_{n}$ as a bi-unary semigroup. It is known from
\cite{shain1970,Stokes2009} that a right restriction Ehresmann semigroup
embeds in $\PT_{n}$ as a bi-unary semigroup if and only if it satisfies
the quasiidentity $xz=yz\implies xz^{+}=yz^{+}$. We now give an example
of an Ehresmann right restriction EI-semigroup that does not satisfy
this quasiidentity.
\begin{example}
\label{exa:Kinyon_Counterexample}\footnote{We thank Professor Michael Kinyon for this example, obtained using
Prover9\cite{Prover9}.} Recall that $\B_{n}$ is the monoid of all binary relations on an
$n$-element set. Denote by $\id_{A}$ the partial identity relation
on the set $A$. Choosing the partial identities as projections, it
is well known that $\B_{n}$ is Ehresmann with $\alpha^{+}=1_{\im(a)}$
and $a^{\ast}=1_{\dom(a)}$ for every $\alpha\in\B_{n}$. Let $S$
be the subsemigroup of $\B_{3}$ consisting of the following $5$
elements:
\begin{align*}
e & =\id_{\{1,2\}},\quad f=\id_{\{3\}},\quad a=\{(3,1),(3,2)\},\quad g=\{(1,2),(2,1)\}
\end{align*}
and the empty relation $0$. It easy to check that $S$ is a bi-unary
subsemigroup of $\B_{3}$ and hence $E$-Ehresmann for $E=\{e,f,0\}$.
It is a routine to verify that $S$ is also right restriction (in
fact, it is also left restriction) and it follows from \lemref{right_restriction_EU_Ehresmann_characterization}
that it is EI-Ehresmann since every idempotent is a projection. On
the other hand, we have $ea=a=ga$ but $ea^{+}=ee=e\neq g=ge=ga^{+}$
so the quasi-identity $xz=yz\implies xz^{+}=yz^{+}$ does not hold
in $S$.
\end{example}

\subsection{EI-restriction semigroups}

If a semigroup $S$ is both EI-Ehresmann and restriction (i.e., left
and right restriction), the situation reduces to a familiar one. A
restriction semigroup $S$ with the property that $E=E(S)$ is called
weakly ample (see \cite[Section 4]{Gould2010}). 
\begin{lem}
Let $S$ be a finite restriction semigroup. Then $S$ is EI-Ehresmann
if and only if $S$ is weakly ample. 
\end{lem}
\begin{proof}
Clearly $E=E(S)$ implies that $e$ is the unique idempotent of $\Ht(e)$
for every $e\in E$ hence $\Ht(e)$ is a group. In the other direction
\lemref{Right_restriction_Ehresmann_leq} and its dual implies that
for every idempotent $f\in E(S)$ we have $f^{+}\leq f^{\ast}$ and
$f^{\ast}\leq f^{+}$ hence $f^{+}=f^{\ast}$. This says that every
idempotent $f$ is $\Ht$-equivalent to a projection $e\in E$ and
being EI-Ehremsann implies $f=e$. 
\end{proof}

\subsection{The infinite case}

In this paper we focus on finite semigroups. However all the results
in \secref{Classes_of_EU_Ehresmann} can be adapted to infinite semigroups
as well. A monoid is called \emph{unipotent }if the identity is its
only idempotent and a category is called EU if all its endomorphisms
monoids are unipotent. We say that an Ehresmann semigroup $S$ is
EU-Ehresmann if its associated Ehresmann category is an EU-category
(i.e., $\Ht(e)$ is a unipotent monoid for every $e\in E$). Clearly
a finite monoid is unipotent if and only if it is a group and a finite
category is EU if and only if it is EI. The arguments in the previous
subsections prove the following statements (where now the semigroups
might be infinite):
\begin{prop}
The class of all EU-Ehresmann semigroups is a quasivariety of bi-unary
semigroups defined by the identities of Ehresmann semigroups and the
quasiidentity 
\[
\text{\mbox{\ensuremath{f^{2}=f\to f^{\ast}ff^{+}=f^{\ast}f^{+}}}}.
\]
.
\end{prop}
\begin{prop}
The class of all right restriction EU-Ehresmann semigroups is a quasivariety
of bi-unary semigroups defined by the identities of Ehresmann semigroups,
the right ample identity $b^{\ast}a=a(b^{\ast}a)^{\ast}$ and the
quasiidentity $f^{2}=f\to ff^{+}=f^{+}$.
\end{prop}
We denote by $\PT_{X}$ the monoid of all partial functions on the
(possible infinite) set $X$.
\begin{prop}
Let $S$ be a regular right restriction EU-Ehresmann semigroup. There
is an embedding of bi-unary semigroups $\Phi:S\to\PT_{S}$ .
\end{prop}
\begin{lem}
Let $S$ be a restriction semigroup. Then $S$ is EU-Ehresmann if
and only if $S$ is weakly ample. 
\end{lem}

\section{\label{sec:Simple_modules}Simple modules of finite right restriction
EI-Ehresmann semigroups}

From now on the focus will be on the case of finite right restriction
EI-Ehresmann semigroups. In this section we would like to describe
the simple modules of algebras of such semigroups and in certain cases
also the indecomposable injective ones. We remark that certain observations
on the semisimple image of such semigroup algebras and their ordinary
quiver can be found in \cite[Proposition 5.17]{Stein2017} and \cite[Section 6.3]{Margolis2012}. 
\begin{lem}
\label{lem:Only_Idempotent_In_L_class}Let $S$ be a finite right
restriction EI-Ehresmann semigroup and let $e\in E$. Then $e$ is
the only idempotent in $\Lc(e)$. 
\end{lem}
\begin{proof}
Assume $f\in\Lc(e)$ is an idempotent so $ef=e$ and $fe=f$. \lemref{right_restriction_EU_Ehresmann_characterization}
implies that $ff^{+}=f^{+}$ and clearly $f^{+}f=f$. Now,

\begin{align*}
e & =ef=ef^{+}f=f^{+}ef=f^{+}e\\
f^{+} & =ff^{+}=fef^{+}=ff^{+}e=f^{+}e
\end{align*}
so $e=f^{+}$. Therefore $f\Lc e$ and $f\Rc f^{+}=e$ so $f\Hc e$
which implies $f=e$. 
\end{proof}
\begin{prop}
Let $S$ be a finite right restriction EI-semigroup. Let $J$ be a
regular $\Jc$-class of $S$. Then the sandwich matrix of $J$ is
left invertible over $\Bbbk G_{J}$ for every field $\Bbbk$. 
\end{prop}
\begin{proof}
Let $E_{J}=E\cap J=\{e_{1},\ldots,e_{k}\}$ be the set of projections
which belongs to $J$. An $\Rc$-class cannot have two projections
and a regular $\Rc$-class has at least one by \lemref{right_restriction_EU_Ehresmann_characterization}
so $J$ has precisely $k$ $\Rc$-classes. Moreover, $e_{1},\ldots,e_{k}$
are also in different $\Lc$-classes. Fix $e=e_{1}$ and choose representatives
$\rho_{1},\ldots,\rho_{k}$ for the $\Hc$-classes of $\Lc(e)$ and
$\lambda_{1},\ldots,\lambda_{l}$ for the $\Hc$-classes of $\Rc(e)$
(note that $l\geq k$). Since projections of $E_{J}$ are in different
$\Lc$ and $\Rc$-classes, we can order the representatives such that
$e_{i}\in\Rc(\rho_{i})\cap\Lc(\lambda_{i}).$ \lemref{Only_Idempotent_In_L_class}
implies that there is no idempotent in $\Rc(\rho_{i})\cap\Lc(\lambda_{j})$
for $i\neq j$ and $1\leq i,j\leq k$ since $e_{j}$ is the unique
idempotent in $\Lc(\lambda_{j})$ (note that there can be idempotents
in $\Rc(\rho_{i})\cap\Lc(\lambda_{j})$ for $j>k$). Therefore, the
sandwich matrix is a $l\times k$ matrix whose upper $k\times k$
block is a diagonal matrix and all the entries on the main diagonal
are non-zero. Therefore, the sandwich matrix is left invertible over
$\Bbbk G_{J}$. 
\end{proof}
If the sandwich matrix of a regular $\Jc$-class $J$ is left invertible
then the module induced from a left Sch\"{u}tzenberger module of $J$
and $V\in\Irr\Bbbk G_{J}$ is simple, see \cite[Corollary 4.22]{Steinberg2016}
and \cite[Lemma 5.20]{Steinberg2016}. This fact and the Clifford-Munn-Ponizovskii
theorem yields the following corollary. 
\begin{thm}
Let $S$ be a finite right restriction EI-Ehresmann semigroup and
let $\Bbbk$ be a field. Let $J$ be a regular $\Jc$-class and let
$V\in\Irr\Bbbk G_{J}$. Then 
\[
\Ind_{G_{J}}(V)=\Bbbk\Lc(e)\bigotimes_{\Bbbk G_{J}}V
\]
\label{thm:Simple_modules_Thm}(for some idempotent $e\in J$) is
a simple module of $\Bbbk S$. In fact, 
\[
\{\Ind_{G_{J}}(V)\mid\text{\text{\ensuremath{J} is a regular \ensuremath{\Jc}-class,\quad\ensuremath{V\in\Irr\Bbbk G_{J}\}}}}
\]
is a list of all the simple modules of $\Bbbk S$ up to isomorphism
with no isomorphic copies. 
\end{thm}
\begin{rem}
Assume in addition that $S$ is regular and that the orders of all
subgroups of $S$ are invertible in $\Bbbk$. Then the left invertibility
of the sandwich matrices implies that the indecomposable injective
modules of $S$ are given by the co-induced modules for regular $\Jc$
classes and $V\in\Irr G_{J}$ (see \cite[Theorem 4.4]{Margolis2011}).
The co-induced module is define by 
\[
\Coind_{G_{J}}(V)=\Hom_{\Bbbk G_{J}}(\Bbbk\Rc(e),V)
\]
(for some idempotent $e\in J$). In this case $\Coind_{G_{J}}(V)$
is the injective envelope of $\Ind_{G_{J}}(V)$. In addition, it follows
that in this case the global dimension of the algebra of $S$ is bounded
by one less than the longest chain of $\Jc$-classes of $S$. The
above results on indecomposable injective modules are true even if
we replace the requirement of regularity by being left Fountain, see
\cite[Theorem 4.8]{Margolis2018B}. 
\end{rem}

\section{\label{sec:Construction_of_Lt_modules}Construction of $\Lt$ - modules}

Our next goal is to obtain a description of indecomposable projective
modules of finite right restriction EI-Ehresmann semigroups. The description
involves an induction using the $\Lt$ relation instead of $\Lc$.
In this intermediate section we consider several cases where such
a construction yields a well defined $\Bbbk S$-module structure.

Let $S$ be a fixed semigroup and let $E\subseteq S$ be a subset
of idempotents. 
\begin{lem}
\label{lem:S_Acts_on_the_left}Let $a\in S$ and denote by $\Lt(a)$
the $\Lt$-class of $a$. Then $S$ acts by partial functions on $\Lt(a)$
according to 
\[
s\cdot x=\begin{cases}
sx & sx\in\Lt(a)\\
\text{undefined} & \text{otherwise}
\end{cases}
\]
for $s\in S$ and $x\in\Lt(a)$. 
\end{lem}
\begin{proof}
The set $L_{1}=\{x\in S\mid x\LeqLt a\}$ is a left ideal of $S$.
For if $x\in L_{1}$ and $s\in S$ then 
\[
ae=a\implies xe=x\implies sxe=sx
\]
for every $e\in E$, so $sx\in L_{1}$ as well. The set $L_{2}=\{x\in S\mid x\LneqqLt a\}$
is also a left ideal of $S$. Indeed, assume $x\in L_{2}$ but $sx\Lt a$
for some $s\in S$. Then 
\[
xe=x\implies sxe=sx\implies ae=a
\]
for every $e\in E$, so $x\Lt a,$ a contradiction. Therefore, $S$
acts by left multiplication on $L_{1}$ and $L_{2}$. It is now clear
that the required action is the Rees quotient $L_{1}/L_{2}$. 
\end{proof}
Now, choose a projection $e\in E$. Recall that we denote the group
$\Hc$-class of $e$ by $G_{e}$. In general, right multiplication
does not induce a right action of the group $G_{e}$ on $\Lt(e)$.
For instance consider the following example. 
\begin{example}
Let $S=\PT_{2}$ and recall that we compose from right to left. Choose
\[
E=\{\id=\id_{\{1,2\}},\id_{\{1\}},\id_{\varnothing}\}
\]
 and $\text{\mbox{\ensuremath{e=\id}}}$ is the identity function.
We have $\id_{\{2\}}\Lt\id$ but we can take the transposition $\mbox{\ensuremath{(12)\in G_{e}=S_{2}}}$
and acting on the right we obtain 
\[
\id_{\{2\}}(12)=\left(\begin{array}{cc}
1 & 2\\
2 & \varnothing
\end{array}\right)
\]
which is not in the $\Lt$-class of $\id$. 
\end{example}
We want to consider cases where the group $G_{e}$ acts on the right
of $\Lt(e)$ by right multiplication. 
\begin{defn}
A semigroup $S$ is called \emph{right Fountain} if every $\widetilde{\mathcal{L}}_{E(S)}$
class contains an idempotent, where $E(S)$ is the set of all idempotents
of $S$. 
\end{defn}
\begin{prop}[{{\cite[Corollary 3.4]{Margolis2018B}}}]
Let $S$ be a finite right Fountain semigroup and let $e\in E(S)$
be an idempotent. Then $G_{e}$ acts on the right of $\widetilde{\mathcal{L}}_{E(S)}(e)$. 
\end{prop}
Another case is where $\Lt$ is a right congruence. 
\begin{lem}
\label{lem:G_Acts_On_The_right_Where_right_cong}Let $S$ be a semigroup
such that $\Lt$ is a right congruence. Then $G_{e}$ acts on the
right of $\Lt(e)$ by right multiplication for every $e\in E$. 
\end{lem}
\begin{proof}
Let $x\in\Lt(e)$ and $g\in G_{e}$. Then $x\Lt e$ so the right congruence
property implies $xg\Lt eg$. Now $g\Hc e$ implies $g\Lt e$ and
$eg=g$ so we have also $xg\Lt e$ as required. 
\end{proof}
\begin{example}
Another trivial case is where $S$ is an $\Hc$-trivial semigroup
hence $G_{e}=\{e\}$ and clearly acts trivially on $\Lt(e)$. 
\end{example}
Recall that $\Bbbk\Lt(e)$ is the $\Bbbk$-vector space with basis
the elements of $\Lt(e)$. The following is an immediate corollary
of \lemref{S_Acts_on_the_left}. 
\begin{cor}
If $G_{e}$ acts on the right of $\Lt(e)$ as in the above cases,
then $\Bbbk\Lt(e)$ is a $\Bbbk S-\Bbbk G_{e}$ bimodule. Therefore,
if $V$ is a $\Bbbk G_{e}$-module, the tensor product 
\[
\Bbbk\Lt(e)\bigotimes_{\mathbb{\Bbbk}G_{e}}V
\]
is a $\mathbb{\Bbbk}S$-module. 
\end{cor}
\begin{rem}
As mentioned above, if $\Lt$ is replaced by the standard Green's
$\mathcal{L}$ class then the above module is the induced Sch\"{u}tzenberger
module of $J$ and $V$ (where $e$ belongs to the $\Jc$-class $J$)
. 
\end{rem}

\section{\label{sec:ProjectiveModules}Projective modules of right restriction
EI-Ehresmann semigroups}

In this section we study projective modules of finite Ehresmann and
right restriction semigroups. The main result will be in the case
of right restriction EI-Ehresmann semigroups. We start with some lemmas
that will be of later use. 
\begin{lem}
\label{lem:LeftActionNonZero}Let $S$ be an Ehresmann and right restriction
semigroup and let $s,m\in S$. Then, 
\[
s\cdot m\in\Lt(m)\iff m^{+}\leq s^{\ast}
\]
(where $\leq$ is the natural partial order on the subsemilattice
$E$). 
\end{lem}
\begin{proof}
First note that 
\[
s\cdot m\in\Lt(m)\iff(sm)^{\ast}=m^{\ast}.
\]

If $(sm)^{\ast}=m^{\ast}$ then by the right ample identity 
\[
s^{\ast}m=m(s^{\ast}m)^{\ast}
\]
and by the right congruence identity 
\[
m(s^{\ast}m)^{\ast}=m(sm)^{\ast}=mm^{\ast}=m.
\]
Since $m^{+}$ is the minimal projection which is a left identity
of $m$ we obtain $\mbox{\ensuremath{m^{+}\leq s^{\ast}}}$. In the
other direction, assume $m^{+}\leq s^{\ast}$ which implies $s^{\ast}m=m.$
By the right ample and right congruence identities we have 
\[
m=s^{\ast}m=m(s^{\ast}m)^{\ast}=m(sm)^{\ast}
\]
which implies that $m^{\ast}\leq(sm)^{\ast}$. However, $smm^{\ast}=sm$
so $(sm)^{\ast}\leq m^{\ast}$ hence $(sm)^{\ast}=m^{\ast}$ as required. 
\end{proof}
\begin{defn}
Let $\leq$ be a poset on a set $S$ and let $s\in S$. The \emph{principal
down ideal} generated by $s$ is the set $s\downarrow=\{x\in S\mid x\leq s\}$. 
\end{defn}
Now consider the poset $\leq_{l}$ on an Ehresmann semigroup $S$.
Recall that on the set of projections $E$, the poset $\leq_{l}$
coincide with the natural semilattice order denoted by $\leq$. Therefore
\[
s\downarrow=\{x\in S\mid x\leq_{l}s\}
\]
and 
\[
s^{\ast}\downarrow=\{e\in E\mid e\leq s^{\ast}\}
\]
for every $s\in S$.

The following lemma is well-known and we prove it for the sake of
completeness. 
\begin{lem}
\label{lem:PosetIsoTheorem}Let $S$ be an Ehresmann semigroup and
let $s\in S$. The function $F:s^{\ast}\downarrow\to s\downarrow$
defined by 
\[
f(e)=se
\]
is an isomorphism of posets with inverse $G:s\downarrow\to s^{\ast}\downarrow$
defined by $G(x)=x^{\ast}$. 
\end{lem}
\begin{proof}
First note that $e_{1}\leq e_{2}$ implies $se_{1}\leq_{l}se_{2}$
since $se_{2}e_{1}=se_{1}$ so $F$ is a homomorphism of posets .
Now, $x_{1}\leq_{l}x_{2}$ says that $x_{1}=x_{2}x_{1}^{\ast}$. By
the right congruence identity 
\[
x_{2}^{\ast}x_{1}^{\ast}=(x_{2}^{\ast}x_{1}^{\ast})^{\ast}=(x_{2}x_{1}^{\ast})^{\ast}=x_{1}^{\ast}
\]
so 
\[
G(x_{1})=x_{1}^{\ast}\leq x_{2}^{\ast}=G(x_{2})
\]
hence $G$ is also a homomorphism of posets. Finally, we show that
$F$ and $G$ are inverses of each other: 
\[
G(F(e))=G(se)=(se)^{\ast}=s^{\ast}e=e
\]
since $e\leq s^{\ast}$ and 
\[
F(G(x))=F(x^{\ast})=sx^{\ast}=x
\]
since $x\leq_{l}s$. 
\end{proof}
Lawson \cite{Lawson1991} proved that the category of all Ehresmann
semigroups is isomorphic to the category of all Ehresmann categories
(with an appropriate choice of morphisms). In fact, the function ${\bf C}$
described in \subsecref{EhresmannSgp} is the object part of this
isomorphism. Another link between certain Ehresmann semigroups and
their corresponding categories appear when we consider their algebras
as seen in the following theorem. 
\begin{thm}[{{\cite[Theorem 1.5]{Stein2018erratum}}}]
\label{thm:MainIso}Let $\Bbbk$ be any commutative unital ring,
let $S$ be a finite Ehresmann and right restriction semigroup and
let $\mathcal{C}$ be the corresponding Ehresmann category. There
is an isomorphism of $\mathbb{\Bbbk}$-algebras $\mathbb{\Bbbk}S\simeq\mathbb{\Bbbk}\mathcal{C}$.
Explicit isomorphisms $\varphi:\mathbb{\Bbbk}S\rightarrow\mathbb{\Bbbk}\mathcal{C}$,
$\psi:\Bbbk\mathcal{C}\rightarrow\mathbb{\Bbbk}S$ are defined (on
basis elements) by 
\[
\varphi(s)=\sum_{t\leq_{l}s}C(t)
\]

\[
\psi(C(x))=\sum_{y\leq_{l}x}\mu(y,x)y
\]

where $\mu$ is the M�bius function of the poset $\leq_{l}$ (see
\cite[Chapter 3]{Stanley1997}). 
\end{thm}
\begin{rem}
This theorem is a generalization of several results due to Solomon
\cite{Solomon1967}, Steinberg \cite{Steinberg2006}, Guo and Chen
\cite{Guo2012} and the second author \cite{Stein2016}. It was further
generalized by Wang \cite{Wang2017} to the class of $P$-Ehresmann
and right\textbackslash{}left $P$-restriction semigroups. $P$-Ehresmann
and $P$-restriction semigroups were introduced by Jones in \cite{Jones2012}. 
\end{rem}
We want to use the isomorphism of \thmref{MainIso} in order to study
the projective modules of $\mathbb{\Bbbk}S$ where $S$ is an Ehresmann
and right restriction semigroup. Let $\mathbb{\Bbbk}$ be a field,
and let $e\in E$ be a projection. Recall that we consider the composition
in the Ehresmann category $\mathcal{C}$ ``from right to left''
so a morphism $C(a)$ has domain $a^{\ast}$ and range $a^{+}$. The
idempotent $e$ is an object of $\mathcal{C}$ with identity morphism
$C(e)$. Consider the set $\mathcal{C}\cdot C(e)$ of all the morphisms
whose domain is $e$. The category $\mathcal{C}$ acts on the left
of $\mathcal{C}\cdot C(e)$ by concatenation of morphisms. More precisely,
recall that an action of $\mathcal{C}$ is a functor $\mathcal{F}$
from $\mathcal{C}$ to the category of sets and functions. Here, for
every object $c$ of $\mathcal{C}$ we set $\mathcal{F}(c)=\mathcal{C}(e,c)$
to be the set of all morphisms with domain e and range $c$. For every
morphism $C(m)\in\mathcal{C}(c_{1},c_{2})$ the morphism $\mathcal{F}(C(m)):\mathcal{C}(e,c_{1})\to\mathcal{C}(e,c_{2})$
is defined by left composition by $C(m)$. This action yields a $\mathbb{\Bbbk}\mathcal{C}$-module
denoted by $\mathbb{\Bbbk}\mathcal{C}\cdot C(e)$ whose elements are
linear combinations of morphisms in $\mathcal{C}\cdot C(e)$. Again,
we can describe this module as a functor from $\mathcal{C}$ to the
category of $\Bbbk$-vector spaces and linear transformations, but
we prefer the description via a category algebra. It is clear that
$\mathbb{\Bbbk}\mathcal{C}\cdot C(e)$ is a projective module since
$C(e)$ is an idempotent of $\mathbb{\Bbbk}\mathcal{C}$. Since $\varphi$
of \thmref{MainIso} is an isomorphism, we can view $\mathbb{\Bbbk}\mathcal{C}\cdot C(e)$
also as a $\mathbb{\Bbbk}S$-module whose action $\star$ is defined
by 
\[
s\star C(m)=\varphi(s)\cdot C(m)=\left(\sum_{t\leq_{l}s}C(t)\right)\cdot C(m)
\]
where $s,m\in S$ and $\cdot$ is the action of $\mathbb{\Bbbk}\mathcal{C}$
on $\mathbb{\Bbbk}\mathcal{C}\cdot C(e)$. However, this description
is not natural from the semigroup point of view, so we would like
to describe this module using the structure of the semigroup itself.
Let $\Lt(e)$ be the $\Lt$ -class of $e$. Observe that $m\Lt e$
if and only if the domain of $C(m)$ is $e$. Therefore, the sets
$\Lt(e)$ and $\mathcal{C}\cdot C(e)$ are in one-to-one correspondence.
Furthermore we claim the following. 
\begin{prop}
There is an isomorphism of $\Bbbk S$-modules 
\[
\Phi:\Bbbk\Lt(e)\to\mathbb{\Bbbk}\mathcal{C}\cdot C(e)
\]
defined (on basis elements) by 
\[
\Phi(m)=C(m)
\]
for every $m\in S$. 
\end{prop}
\begin{proof}
It is clear that $\Phi$ is an isomorphism of $\Bbbk$-vector spaces.
For every $s\in S$ it is left to prove that 
\[
\Phi(s\cdot m)=s\star\Phi(m)
\]
where 
\[
s\star\Phi(m)=\varphi(s)\Phi(m)=\sum_{r\leq_{l}s}C(r)\cdot C(m).
\]
First we prove that $s\cdot m=0\iff{\displaystyle \sum_{r\leq_{l}s}}C(r)\cdot C(m)=0$.
According to \lemref{LeftActionNonZero}, $s\cdot m=0$ implies $m^{+}\not\leq s^{\ast}$
so by \lemref{PosetIsoTheorem} there is no element $t\leq_{l}s$
such that $t^{\ast}=m^{+}$ hence 
\[
\sum_{r\leq_{l}s}C(r)\cdot C(m)=0.
\]
The converse implication follows in the same manner. Now, assume $s\cdot m\neq0$
so $m^{+}\leq s^{\ast}$. By \lemref{PosetIsoTheorem}, there exists
a unique $t\leq_{l}s$ with $t^{\ast}=m^{+}$. In other words, there
is a unique $t\leq_{l}s$ such that $\text{\mbox{\ensuremath{C(t)\cdot C(m)\neq0}}}$.
Therefore, 
\[
s\star\Phi(m)=\sum_{r\leq_{l}s}C(r)\cdot C(m)=C(t)\cdot C(m)=C(tm).
\]
Note that $tm=st^{\ast}m=sm^{+}m=sm$ hence 
\[
\Phi(sm)=C(sm)=C(tm)=s\star\Phi(m)
\]
as required. 
\end{proof}
Let $G_{e}$ be the group $\mathcal{H}$-class of $e$ and let $g\in G_{e}$.
According to \lemref{G_Acts_On_The_right_Where_right_cong} ,$G_{e}$
acts on the right of $\Lt(e)$ since $\Lt$ is a right congruence.
It is also clear that the domain and range of $C(g)$ is $e$, so
$G$ acts on the right of $\mathcal{C}\cdot C(e)$ as well. Moreover,
for every $m\in\Lt(e)$ we have 
\[
C(mg)=C(m)C(g)
\]
so $\Phi$ is not only a $\Bbbk S$ -modules isomorphism but a $\Bbbk S-\Bbbk G_{e}$
bimodule isomorphism. Therefore, we obtain the following corollary. 
\begin{cor}
\label{cor:Proj_Modules_Form}Let $S$ be a finite Ehresmann and right
restriction semigroup and let $e\in E$. Denote by $G_{e}$ the group
$\mathcal{H}$-class of $e$ and let $V$ be any $G_{e}$-module.
Then there is an isomorphism of $\Bbbk$S modules 
\[
\Bbbk\Lt(e)\bigotimes_{\mathbb{\Bbbk}G_{e}}V\simeq\mathbb{\Bbbk}\mathcal{C}\cdot C(e)\bigotimes_{\Bbbk G_{e}}V.
\]
\end{cor}
In particular, if $p\in\Bbbk G_{e}$ is a primitive idempotent of
$\Bbbk G_{e}$, and $C(p)$ is the corresponding linear combination
of morphisms, there is an isomorphism 
\[
\mathbb{\Bbbk}\mathcal{C}C(p)=\mathbb{\Bbbk}\mathcal{C}C(e)C(p)\simeq\mathbb{\Bbbk}\mathcal{C}\cdot C(e)\bigotimes_{\Bbbk G_{e}}\Bbbk G_{e}C(p)\simeq\Bbbk\Lt(e)\bigotimes_{\mathbb{\Bbbk}G_{e}}\Bbbk G_{e}p.
\]
Clearly, $\mathbb{\Bbbk}\mathcal{C}C(p)$ is a projective module since
$C(p)$ is an idempotent. So this gives a semigroup theoretic description
of certain projective modules of $S$. However, there is no reason
for the indecomposable projective modules to be of this form. The
situation is much better if we consider (finite) right restriction
EI-Ehresmann semigroups.

A proof of the following lemma can be found in \cite[Lemma 9.31]{Luck1989}
or \cite[Corollary 4.5]{Webb2007}. 
\begin{lem}
\label{lem:EquivalenceOfPrimitiveIdempotents}Let $\mathcal{C}$ be
a finite EI-category and given an object $e$ let $G_{e}=\mathcal{C}(e,e)$
to be its automorphism group and $P_{e}=\{p_{1}^{e},\ldots,p_{m_{e}}^{e}\}$
to be a complete set of primitive orthogonal idempotents for $\Bbbk G_{e}$.
The set 
\[
\bigcup_{e}P_{e}
\]
(where the union is taken over all objects of $\mathcal{C}$) is a
complete set of primitive orthogonal idempotents for $\Bbbk\mathcal{C}$. 
\end{lem}
From now on we assume that the order of every (maximal) subgroup of
$S$ is invertible in $\Bbbk$. This implies that $\Bbbk G_{e}$ is
semisimple for every $e\in E$ by Maschke's theorem. 
\begin{thm}
\label{thm:ProjectiveModulesDescription}Let $S$ be a finite right
restriction EI-Ehresmann semigroup. Let $e\in E$ and $G_{e}=\mathcal{C}(e,e)$
be as above and let $V\in\Irr\Bbbk G_{e}$. Then 
\[
\Bbbk\Lt(e)\bigotimes_{\mathbb{\Bbbk}G_{e}}V
\]
is an indecomposable projective module of $\Bbbk S$. Moreover, every
indecomposable projective module of $\Bbbk S$ is isomorphic to a
module of this form for an appropriate choice of a projection $e\in E$
and a simple $G_{e}$-module $V$. 
\end{thm}
\begin{proof}
This is immediate from \corref{Proj_Modules_Form} and the fact that
if $\mathbb{\Bbbk}G_{e}$ is semisimple then $V\simeq\mathbb{\Bbbk}G_{e}p$
where $p$ is a primitive idempotent of $\mathbb{\Bbbk}G_{e}$ which
is also a primitive idempotent of $\mathbb{\Bbbk}\mathcal{C}$ by
\lemref{EquivalenceOfPrimitiveIdempotents}. 
\end{proof}
\begin{rem}
A similar result for indecomposable projective modules of a certain
type of right Fountain monoids is obtained in \cite[Theorem 4.7]{Margolis2018B}. 
\end{rem}

\paragraph{Isomorphic copies of indecomposable projective modules}

The set 
\[
\{\Bbbk\Lt(e){\displaystyle \bigotimes_{\mathbb{\Bbbk}G_{e}}}V\mid e\in E,\quad V\in\Irr G_{e}\}
\]
usually contains isomorphic modules. We want to obtain a list of all
indecomposable projective modules up to isomorphism but without isomorphic
copies. From the categorical point of view we can say that 
\[
\Bbbk\Lt(e_{1}){\displaystyle \bigotimes_{\mathbb{\Bbbk}G_{e_{1}}}}V_{1}\simeq\Bbbk\Lt(e_{2}){\displaystyle \bigotimes_{\mathbb{\Bbbk}G_{e_{2}}}}V_{2}
\]
if and only if $e_{1}$ is isomorphic to $e_{2}$ (as objects in a
category) and $V_{1}\simeq V_{2}$ as $G_{e_{1}}$-modules (note that
$G_{e_{1}}\simeq G_{e_{2}}$) - see \cite[Corollary 4.2 and Proposition 4.3]{Webb2007}.
From the semigroup point of view we can say even more. We start with
two lemmas that will make the situation more transparent. 
\begin{lem}[{{\cite[Corollary 5.5]{Stein2017}}}]
\label{lem:Isomorphic_Objects_in_Ehresmann_Cat} Let $S$ be a finite
Ehresmann semigroup and let $e,f\in E$ be two projections. Then $e$
and $f$ are isomorphic (as objects in the corresponding Ehresmann
category) if and only if $e\Jc f$ in $S$. 
\end{lem}
\begin{lem}
\label{lem:EI_Ehresmann_Implies_Projection_in_every_J_class}Let $S$
be a finite EI-Ehresmann semigroup. Then in every regular $\Jc$-class
of $S$ there is a projection $e\in E$. 
\end{lem}
\begin{proof}
Let $J$ be a regular $\Jc$-class and choose an idempotent $f\in J$.
\lemref{EU_Ehresmann_characterization} says that $f^{\ast}f^{+}$
is an inverse of $f$. Therefore $f\Jc f^{\ast}f^{+}$ so $e=f^{\ast}f^{+}$
is a projection in $J$. 
\end{proof}
\begin{rem}
The converse of \lemref{EI_Ehresmann_Implies_Projection_in_every_J_class}
is not true. Recall that $B_{2}$ is the monoid of all binary relations
on the set $\{1,2\}$, i.e., subsets of $\{1,2\}^{2}$. We have already
mentioned (\exaref{Kinyon_Counterexample}) that this is an Ehresmann
semigroup with respect to the semilattice of partial identities -
$\alpha^{+}=1_{\im(a)}$ and $\alpha^{\ast}=1_{\dom(a)}$ for every
$\alpha\in B_{2}$. Consider the submonoid $M\subseteq B_{2}$ of
the $12$ relations whose sizes are $0,1,2$ or $4$ (as subsets of
$\{1,2\}^{2}).$ It is a routine matter to verify that this is indeed
a submonoid of $B_{2}$. Since $E\subseteq M$ it is clear that $M$
is also an Ehresmann semigroup. It is also routine to check that $M$
has three $\Jc$-classes: One of the two invertible elements, one
with the zero element and another with the other $9$ elements. Therefore,
every $\Jc$-class is regular and contains a projection. On the other
hand, choose the idempotent $f=\{(1,1),(2,2),(1,2),(2,1)\}\in M$.
Since $\dom(f)=\im(f)=\{1,2\}$ it is clear that the morphism $C(f)$
is a non-invertible endomorphism of the object $\id\in E$ in the
corresponding Ehresmann category. Therefore, the category is not an
EI-category and $M$ is not EI-Ehresmann. 
\end{rem}
According to \lemref{Isomorphic_Objects_in_Ehresmann_Cat} and \lemref{EI_Ehresmann_Implies_Projection_in_every_J_class}
choosing one object from every isomorphism class of objects is precisely
like choosing one projection from every regular $\Jc$-class. So we
can fix a set of representative projections $I=\{e_{1},\ldots,e_{k}\}$,
one from every regular $\Jc$-class whose corresponding maximal groups
are $G_{J_{1}},\ldots,G_{J_{k}}$. Now, 
\[
\left\{ \Bbbk\Lt(e_{i}){\displaystyle \bigotimes_{\mathbb{\Bbbk}G_{J_{i}}}}V\mid i=1,\ldots,k\quad V\in\Irr\Bbbk G_{J_{i}}\right\} 
\]
is a list of all the indecomposable projective modules (up to isomorphism)
without isomorphic copies. This is an explicit correspondence between
indecomposable projective modules of $S$ and pairs $(J,V)$ where
$J$ is a regular $\Jc$-class and $\text{\mbox{\ensuremath{V\in\Irr\Bbbk G_{J}}}}$.
From this point of view, the correspondence between a simple module
and its projective cover is clear. It is easy to see that for every
representative projection $e\in I$ there is an $\Bbbk S$-module
epimorphism 
\[
\psi:\Bbbk\Lt(e)\to\Bbbk\Lc(e)
\]
defined on basis elements by 
\[
\psi(m)=\begin{cases}
m & m\in\Lc(e)\\
0 & m\notin\Lc(e).
\end{cases}
\]
This induces an $\Bbbk S$-module epimorphism 
\[
\psi\otimes\id:\Bbbk\Lt(e){\displaystyle \bigotimes_{\mathbb{\Bbbk}G_{J}}}V\to\Bbbk\Lc(e){\displaystyle \bigotimes_{\mathbb{\Bbbk}G_{J}}}V
\]

for every $V\in\Irr\Bbbk G_{J}$ (where $J$ is the regular $\Jc$-class
of $e$). This is an explicit description of the epimorphism from
an indecomposable projective module and its simple image.

\paragraph{Dimension of indecomposable projective modules}

Next we want to describe the dimension of the projective modules.
Recall that we assume that the order of every group being discussed
is invertible in $\Bbbk$ hence its algebra is semisimple and here
we also assume that $\Bbbk$ is algebraically closed. Recall that
if $V$ is a $\Bbbk G$-module then the dual $D(V)=\Hom_{\Bbbk}(V,\mathbb{\Bbbk})$
is a right $\Bbbk G$-module. We will use a well-known fact whose
proof we will briefly sketch. 
\begin{prop}
\label{prop:Dimension_Of_Module_Of_a_Group}Let $G$ be a finite group,
let $M$ be a right $\Bbbk G$-module and let $V\in\Irr\mathbb{\Bbbk}G$.
The dimension of ${\displaystyle M\bigotimes_{\Bbbk G}V}$ (as a $\mathbb{\Bbbk}$-vector
space) is the multiplicity of $D(V)$ in the decomposition of $M$
into simple (right) $\mathbb{\Bbbk}G$-modules. 
\end{prop}
\begin{proof}
Let $V_{1},\ldots,V_{r}$ be the simple $\mathbb{\Bbbk}G$-modules
up to isomorphism. Fix a set $\{p_{1},\ldots,p_{r}\}$ of primitive
orthogonal idempotents such that $V_{i}\simeq\mathbb{\Bbbk}Gp_{i}$.
It is easy to check that 
\[
p_{i}\mathbb{\Bbbk}G\bigotimes_{\mathbb{\Bbbk}G}\mathbb{\Bbbk}Gp_{j}=p_{i}\mathbb{\Bbbk}Gp_{j}
\]
as $\Bbbk$-vector spaces. Moreover, since $\Bbbk$ is algebraically
closed, the Wedderburn-Artin theorem implies that 
\[
p_{i}\mathbb{\Bbbk}Gp_{j}\simeq\begin{cases}
\mathbb{\Bbbk} & i=j\\
0 & \text{otherwise}
\end{cases}
\]
hence, 
\[
\dim p_{i}\mathbb{\Bbbk}Gp_{j}=\begin{cases}
1 & i=j\\
0 & \text{otherwise}
\end{cases}
\]
(this is actually Schur's Lemma). Now, consider the decomposition
of $M$ into a sum of simple right $\mathbb{\Bbbk}G$-modules 
\[
M\simeq\bigoplus_{l=1}^{r}n_{l}p_{l}\mathbb{\Bbbk}G
\]
and let $V\simeq\mathbb{\Bbbk}Gp_{i}$ for some $i$. It is clear
that 
\[
\dim M\bigotimes_{\mathbb{C}G}\mathbb{\Bbbk}Gp_{i}=\dim\bigoplus_{l=1}^{r}n_{l}p_{l}\mathbb{\Bbbk}Gp_{i}=n_{i}.
\]
So the dimension is indeed the multiplicity of $D(V)\simeq p_{i}\mathbb{\Bbbk}G$
in the decomposition of $M$ as required. 
\end{proof}
As an immediate corollary of \thmref{ProjectiveModulesDescription}
and \propref{Dimension_Of_Module_Of_a_Group}, we obtain the following
result. 
\begin{cor}
\label{cor:DimensionProjective}Let $S$ be a finite right restriction
EI-Ehresmann semigroup. Let $J$ be a regular $\Jc$-class and choose
a projection $e\in J$ (such exists by \lemref{EI_Ehresmann_Implies_Projection_in_every_J_class}).
Let $V\in\Irr\mathbb{\Bbbk}G_{J}$ where $G_{J}$ is the maximal subgroup
associated with $J$. Assume also that $\Bbbk$ is algebraically closed
and the order of every subgroup of $S$ is invertible in $\Bbbk$.
Then the dimension of $\mbox{\ensuremath{{\displaystyle \mathbb{\Bbbk}\tilde{L}_{E}(e)\bigotimes_{\mathbb{C}G_{J}}V}}}$
as a $\Bbbk$-vector space is the multiplicity of $D(V)$ in the decomposition
of $\mathbb{\Bbbk}\Lt(e$) as a right $G_{J}$-module. 
\end{cor}

\section{\label{sec:PartialFunctions}The monoid of partial functions}

Recall that $\PT_{n}$ denotes the monoid of all partial functions
on the set $\{1,\ldots,n\}$. In this section we apply the results
of \secref{ProjectiveModules} in this specific case. For this section
we fix $\Bbbk=\mathbb{C}$, the field of complex numbers. Let $A\subseteq\{1,\ldots,n\}$
and denote by $\id_{A}$ the partial identity function on $A$. As
already mentioned in \subsecref{Right_restriction_EU_Ehresmann},
$\PT_{n}$ is a right restriction EI-Ehresmann semigroup with 
\[
\text{\mbox{\ensuremath{E=\{\id_{A}\mid A\subseteq\{1,\ldots,n\}\}}}}
\]
as a subsemilattice of projections. Recall that we are composing functions
from right to left. We denote the corresponding Ehresmann category
by $E_{n}$, and it can be described in the following way. The objects
of $E_{n}$ are subsets of $\{1,\ldots,n\}$ and for every $A,B\subseteq\{1,\ldots,n\}$
the hom-set $E_{n}(A,B)$ contains all the (total) onto functions
from $A$ to $B$. The endomorphism monoid of an object $A$ is the
symmetric group $S_{A}$ so the category $E_{n}$ is indeed an EI-category.
Let $f:A\to B$ be a partial function. We denote by $\dom(f)\subseteq A$
and $\im(f)\subseteq B$ the domain and image of $f$. Recall that
the (set theoretic) kernel of $f$ is the equivalence relation on
$\dom f$ defined by $a_{1}\sim a_{2}$ if $f(a_{1})=f(a_{2}).$ We
denote it by $\ker f$. Consider two partial functions $f_{1},f_{2}\in\PT_{n}$.
It is easy to see that $f_{1}\Lt f_{2}$ if and only if $\dom(f_{1})=\dom(f_{2})$.
It is well known that $f_{1}$ and $f_{2}$ are $\Jc$-equivalent
if $|\im(f_{1})|=|\im(f_{2})|$ and $\Lc$-equivalent if $\dom(f_{1})=\dom(f_{2})$
and $\ker f_{1}=\ker f_{2}$. All the $\Jc$-classes of $\PT_{n}$
are regular and we choose a representative projection $\id_{k}=\id_{\{1,\ldots,k\}}\in E$
for every $\Jc$-class. The group corresponding to the $\Jc$-class
of $\id_{k}$ is the symmetric group $S_{k}$. Therefore, the maximal
subgroups of $\PT_{k}$ are $S_{k}$ for $0\leq k\leq n$ (note that
$S_{0}\simeq S_{1}$). It is well known that simple modules of $\mathbb{C}S_{k}$
are indexed by partitions of $k$ (or Young diagrams). Therefore,
the simple modules of $\mathbb{\mathbb{C}}\PT_{n}$ can be indexed
by partitions $\lambda\vdash k$ for $0\leq k\leq n$. We denote the
simple module of $S_{k}$ corresponding to a partition $\lambda\vdash k$
by $S^{\lambda}$. According to \thmref{Simple_modules_Thm} all simple
modules of $\mathbb{C}\PT_{n}$ are of the form ${\displaystyle \mathbb{\mathbb{C}}\Lc(\id_{k})\bigotimes_{\mathbb{C}S_{k}}S^{\lambda}}$
for $0\leq k\leq n$ and $\lambda\vdash k$ (this is a known fact,
see \cite[Section 11.3]{Ganyushkin2009b}, and this is also a corollary
of \cite[Theorem 4.4]{Margolis2011}). \thmref{ProjectiveModulesDescription}
yields the following result. 
\begin{thm}
The modules ${\displaystyle \mathbb{\mathbb{C}}\Lt(\id_{k})\bigotimes_{\mathbb{C}S_{k}}S^{\lambda}}$
for $0\leq k\leq n$ and $\lambda\vdash k$ form a list of all the
indecomposable projective modules of $\mathbb{C}\PT_{n}$ up to isomorphism.
(Note that $\Lt(\id_{k})$ consists of all the partial functions $f\in\PT_{n}$
with $\dom(f)=\{1,\ldots,k\}$). 
\end{thm}

\paragraph{Dimension of indecomposable projective modules}

Our goal now is to obtain a formula for the dimension (as a $\mathbb{C}$-vector
space) of the projective module $\mathbb{C}\Lt(\id_{k})\otimes S^{\lambda}$
where $\lambda\vdash k$. By \corref{DimensionProjective} we need
to consider the multiplicity of $D(S^{\lambda})$ in the right $\mathbb{C}S_{k}$
module $\mathbb{C}\Lt(\id_{k})$. This is a permutation module, induced
by the (right) action of $S_{k}$ on $\Lt(\id_{k})$. We start by
examining this action. Let $f:\{1,\ldots,k\}\to\{b_{1},\ldots,b_{r}\}$
be an onto function. We associate with $f$ a composition $\mu_{f}=[\mu_{1},\ldots,\mu_{r}]\vDash k$
where $\mu_{i}$ is the number of elements in the preimage $f^{-1}(b_{i}).$
For instance, if $f:\{1,2,3,4,5,6\}\to\{2,3,6\}$ is defined by 
\[
f=\left(\begin{array}{cccccc}
1 & 2 & 3 & 4 & 5 & 6\\
2 & 6 & 2 & 6 & 3 & 6
\end{array}\right)
\]
then $\mu_{f}=[2,1,3]\vDash6$. It is clear that two functions $f_{1},f_{2}\in\Lt(\id_{k})$
are in the same orbit (under the right action of $S_{k}$) if and
only if $\im(f_{1})=\im(f_{2})$ and $\mu_{f_{1}}=\mu_{f_{2}}$. Therefore,
we can index the orbits by pairs $(B,\mu$) where $B\subseteq\{1,\ldots,n\}$
and $\mu\vDash k$ is a composition with $|\mu|=|B|$. Denote such
an orbit by $O_{(B,\mu)}.$ It is also easy to check that if $f\in O_{(B,\mu)}$
with $\mu=[\mu_{1},\ldots,\mu_{r}]$ then the stabilizer of $f$ is
isomorphic to $S_{\mu}=S_{\mu_{1}}\times\ldots\times S_{\mu_{r}}$
(the Young subgroup corresponding to $\mu$). Now, as a $\mathbb{C}$-vector
space we have a decomposition 
\begin{align*}
\mathbb{C}\Lt(\id_{k})\bigotimes_{\mathbb{C}S_{k}}S^{\lambda} & \simeq\left(\bigoplus\mathbb{C}O_{(B,\mu)}\right)\bigotimes_{\mathbb{C}S_{k}}S^{\lambda}
\end{align*}
where the sum is over all pairs $(B,\mu)$ where $B\subseteq\{1,\ldots,n\}$
and $\mu\vDash k$ is a composition with $|\mu|=|B|$. This is clearly
isomorphic to 
\[
\bigoplus\left(\mathbb{C}O_{(B,\mu)}\bigotimes_{\mathbb{C}S_{k}}S^{\lambda}\right).
\]
Since $S_{k}$ acts transitively on $O_{(B,\mu)}$ and since $S_{\mu}$
is the stabilizer of this action it is known that $\mathbb{C}O_{(B,\mu)}\simeq\Ind_{S_{\mu}}^{S_{k}}\tr_{\mu}$
where $\tr_{\mu}$ is the trivial module of $\mathbb{C}S_{\mu}$.
This is the Young module corresponding to $\mu$. The multiplicity
of $S^{\lambda}$ in the decomposition of $\Ind_{S_{\mu}}^{S_{k}}\tr_{\mu}$
(or $D(S^{\lambda})$ if we consider the right module) is the Kostka
number $K_{\lambda\mu}$. Therefore, 
\[
\dim\mathbb{C}O_{(B,\mu)}\bigotimes_{\mathbb{C}S_{k}}S^{\lambda}=K_{\lambda\mu}
\]
and 
\[
\dim\bigoplus\left(\mathbb{C}O_{(B,\mu)}\bigotimes_{\mathbb{C}S_{k}}S^{\lambda}\right)=\sum K_{\lambda\mu}
\]
where again the sum is over all pairs $(B,\mu)$ where $B\subseteq\{1,\ldots,n\}$
and $\mu\vDash k$ is a composition with $|\mu|=|B|$. Since $K_{\lambda\mu}$
does not depend on $B$ this equals 
\[
\sum_{l=1}^{k}\sum_{\overset{\mu\vDash k}{|\mu|=l}}{n \choose l}K_{\lambda\mu}.
\]
So we have obtained the following result. 
\begin{prop}
\label{prop:dim_of_proj_module}The dimension of the indecomposable
projective module \linebreak{}
$\text{\mbox{\ensuremath{\mathbb{C}\Lt(\id_{k})\otimes\mathbb{C}S^{\lambda}}}}$
is given by the formula 
\[
\sum_{l=1}^{k}\sum_{\overset{\mu\vDash k}{|\mu|=l}}{n \choose l}K_{\lambda\mu}.
\]
\end{prop}

\paragraph*{Cartan matrix of $\mathbb{C}\PT_{n}$}
\begin{defn}
Let $A$ be a $\mathbb{C}$-algebra. Let $\{S(1),\ldots,S(r)\}$ be
the irreducible representations of $A$ up to isomorphism and let
$\{P(1),\ldots,P(r)\}$ be the indecomposable projective modules of
$A$ ordered such that $P(i)$ is the projective cover of $S(i).$
The \emph{Cartan matrix} of $A$ is the $r\times r$ matrix whose
$(i,j)$ entry is the number of times that $S(i)$ appears as a Jordan-H�lder
factor in $P(j)$.

We want to show that certain entries in the Cartan matrix of $\mathbb{C}\PT_{n}$
equal $0$. The proof is similar to the proof of \propref{dim_of_proj_module}.
The irreducible representations of $\mathbb{C}\PT_{n}$ (and hence
the rows and columns of the Cartan matrix) are indexed by Young diagrams
with $k$ boxes for $0\leq k\leq n$. For every Young diagram $\alpha\vdash k$
we denote by $S(\alpha)$ and $P(\alpha)$ the associated irreducible
and indecomposable projective module respectively. We think of the
Cartan matrix as an $(n+1)\times(n+1)$ block matrix where the $(i,j)$
block contains pairs $(\alpha,\beta)$ of Young diagrams such that
$\alpha\vdash(i-1)$ and $\beta\vdash(j-1)$. Denote by $E(k,r)$
the set of all onto functions $f:\{1,\ldots,k\}\to\{1,\ldots,r\}$.
It is clear that $\mathbb{C}E(k,r)$ has a natural structure of $\mathbb{C}S_{r}-\mathbb{C}S_{k}$
bi-module. The following result regarding the Cartan matrix was obtained
in \cite[Corollary 4.2]{Stein2019} (this is a specific case of a
known formula for the Cartan matrix of any EI-category algebra). 
\end{defn}
\begin{prop}
Let $\alpha\vdash r$ and $\beta\vdash k$ be two Young diagrams.
The number of times that $S(\alpha)$ appears as a Jordan-H�lder factor
of $P(\beta)$ is the number of times that $S^{\alpha}$ appears as
an irreducible constituent in the $\mathbb{C}S_{r}$ module $\mathbb{C}E(k,r){\displaystyle \bigotimes_{\mathbb{C}S_{k}}S^{\beta}}$. 
\end{prop}
\begin{rem}
It is proved in \cite{Stein2019} that the Cartan matrix of $\mathbb{C}\PT_{n}$
is block upper unitriangular. Moreover, a combinatorial description
for the first and second block superdiagonals is given there. 
\end{rem}
\begin{prop}
\label{prop:Module_Is_zero_If_content_long}Let $\beta\vdash k$ be
a Young diagram such that $r<|\beta|$ then $\mathbb{C}E(k,r){\displaystyle \bigotimes_{\mathbb{C}S_{k}}S^{\beta}=0}$. 
\end{prop}
\begin{proof}
According to \propref{Dimension_Of_Module_Of_a_Group}, 
\[
\dim\mathbb{C}E(k,r){\displaystyle \bigotimes_{\mathbb{C}S_{k}}S^{\beta}}
\]
is the multiplicity of $D(S^{\beta})$ in the right $\mathbb{C}S_{k}$-module
$\mathbb{C}E(k,r)$. This is a permutation module, induced from the
right action of $S_{k}$ on $E(k,r)$. It is clear that two functions
$f_{1},f_{2}\in E(k,r)$ are in the same orbit if and only if $\mu_{f_{1}}=\mu_{f_{2}}$
(again, $\mu_{f}=[\mu_{1},\ldots,\mu_{r}]\vDash k$ where $\mu_{i}=|f^{-1}(i)|$).
Therefore, we can index the orbits by compositions, $\mu\vDash k$
with $|\mu|=r$. Denote such an orbit by $O_{\mu}$. As $\mathbb{C}$-Vector
spaces 
\begin{align*}
\mathbb{C}E(k,r){\displaystyle \bigotimes_{\mathbb{C}S_{k}}S^{\beta}} & \simeq\left(\bigoplus_{\underset{|\mu|=r}{\mu\vDash k}}O_{\mu}\right){\displaystyle \bigotimes_{\mathbb{C}S_{k}}S^{\beta}}\\
 & \simeq\bigoplus_{\underset{|\mu|=r}{\mu\vDash k}}\left(O_{\mu}{\displaystyle \bigotimes_{\mathbb{C}S_{k}}S^{\beta}}\right).
\end{align*}
It is clear that the stabilizer of $f\in O_{\mu}$ is $S_{\mu}\simeq S_{\mu_{1}}\times\cdots\times S_{\mu_{r}}.$
As we have already seen, this implies that 
\[
\dim O_{\mu}{\displaystyle \bigotimes_{\mathbb{C}S_{k}}S^{\beta}}=\dim\Ind_{S_{\mu}}^{S_{k}}\tr_{\mu}{\displaystyle \bigotimes_{\mathbb{C}S_{k}}S^{\beta}}=K_{\beta\mu}.
\]
The assumption $|\mu|=r<|\beta|$ implies that $K_{\beta\mu}=0$ since
there can be no semistandard Young tableau with shape $\beta$ and
content $\mu$ for $|\mu|<|\beta|$. 
\end{proof}
\propref{Module_Is_zero_If_content_long} immediately implies the
following corollary: 
\begin{prop}
\label{prop:Zero_entries_thm}Let $\alpha\vdash r$ and $\beta\vdash k$
two young diagrams such that $r<|\beta|.$ Then the $(\alpha,\beta)$
entry in the Cartan matrix of $\mathbb{C}\PT_{n}$ is $0$. 
\end{prop}
\begin{rem}
This result is a generalization of some known facts about the projective
modules of $\mathbb{C}\PT_{n}$. Consider the case where $\beta=[1^{k}]$
for $1\leq k\leq n$. \propref{Zero_entries_thm} (with the fact that
the Cartan matrix is unitriangular) says that the only irreducible
constituent in $P(\beta)$ is $S(\beta)$. Therefore $S(\beta)\simeq P(\beta)$
which is a known fact about $\PT_{n}$ (see \cite[Lemma 6.7]{Stein2019}).
In the case where $\beta=[2,1^{k-2}]$ for some $2\leq k\leq n$,
\propref{Zero_entries_thm} is precisely \cite[Lemma 6.4]{Stein2019}
which is a key step in the proof that the global dimension of $\mathbb{C}\PT_{n}$
is $n-1$. 
\end{rem}

\section*{Appendix: \label{sec:Counterexample}Ehresmann semigroups and categories
- two counterexamples}

The isomorphism between the algebras of a right restriction Ehresmann
semigroup and the corresponding category (\thmref{MainIso}) was presented
in \cite{Stein2017} without the requirement of $S$ being right restriction.
However, Shoufeng Wang has found out that the proof implicitly assumes
it. In fact, the functions $\psi$ and $\varphi$ of \thmref{MainIso}
are $\Bbbk$-algebra homomorphisms if and only if $S$ is right restriction
\cite[Lemma 4.3]{Wang2017}. This led to a correction \cite{Stein2018erratum},
but it was not clear whether being right restriction is necessary
for the conclusion of \thmref{MainIso} to be true. In this appendix
we show that this requirement can not be omitted. A counterexample
is given by the partition monoid which arises in the study of partition
algebras (see \cite{Halverson2005}). On the other hand, we show that
the algebra of the semigroup in \exaref{EI_Ehresmann_not_Right_left_restriction}
is isomorphic to the corresponding category algebra despite being
neither left nor right restriction. 

The partition monoid $\PR_{n}$ for $n\in\mathbb{N}$ can be described
in the following way. Define ${\bf n}=\{1,\ldots,n\}$ and ${\bf n}^{\prime}=\{1^{\prime},\ldots,n^{\prime}\}$.
An element $\alpha\in\PR_{n}$ is a partition of the set ${\bf n}\cup{\bf n}^{\prime}$.
For every $x,y\in{\bf n}\cup{\bf n}^{\prime}$ we write $x\sim_{\alpha}y$
when $x$ and $y$ are in the same partition class of $\alpha$. We
refer the reader to \cite{Halverson2005} for the description of the
operation of $\PR_{n}$ which is best pictured with certain diagrams. 

For any partition $\alpha\in\PR_{n}$ we follow \cite{East2011} and
define two equivalence relations on ${\bf n}$ by
\[
\ker(\alpha)=\{(i,j)\in{\bf n}\times{\bf n}\mid i\sim_{\alpha}j\}
\]
\[
\coker(\alpha)=\{(i,j)\in{\bf n}\times{\bf n}\mid i^{\prime}\sim_{\alpha}j^{\prime}\}.
\]
Note that there are two types of equivalence classes in $\ker(\alpha)$.
A kernel class $K$ is in the domain of $\alpha$ if there exists
$j^{\prime}\in{\bf n}^{\prime}$ such that $j^{\prime}\sim_{\alpha}i$
for some (and hence all) $i\in K$. Likewise a cokernel class $K^{\prime}$
is in the codomain of $\alpha$ if there exists $i\in{\bf n}$ such
that $i\sim_{\alpha}j^{\prime}$ for some (all) $j^{\prime}\in K^{\prime}$.
Note also that the number of domain classes in $\ker(\alpha)$ equals
the number of codomain classes in $\coker(\alpha)$. For every equivalence
relation $\sigma$ on ${\bf n}$ we define a partition $e_{\sigma}\in\PR_{n}$
in the following way. Two elements $x,y\in{\bf n}\cup{\bf n^{\prime}}$
will be in the same $e_{\sigma}$-class if $(i,j)\in\sigma$ where
$x=i$ or $x=i^{\prime}$ and $y=j$ or $y=j^{\prime}$. It is clear
that $e_{\sigma}$ is an idempotent with $\ker(e_{\sigma})=\coker(e_{\sigma})=\sigma$.
Moreover, $e_{\sigma_{1}}e_{\sigma_{2}}=e_{\sigma_{3}}$ where $\sigma_{3}$
is the minimal equivalence relation such that $\sigma_{1},\sigma_{2}\subseteq\sigma_{3}$
(i.e., it is the join in the lattice of equivalence relations on ${\bf n}$).
Therefore, the set 
\[
E=\{e_{\sigma}\mid\sigma\text{ is an equivalence relation on }{\bf n}\}
\]
is a subsemilattice of $\PR_{n}$ (note that $e_{\sigma_{1}}\leq e_{\sigma_{2}}$if
$\sigma_{2}\subseteq\sigma_{1}$). For every $\alpha\in\PR_{n}$ it
is clear that $e_{\ker(\alpha)}$ ($e_{\coker(\alpha)}$) is the minimal
element $e\in E$ such that $e\alpha=\alpha$ (respectively, $\alpha e=\alpha$)
hence we can write 
\[
\alpha^{+}=e_{\ker(\alpha)},\quad\alpha^{\ast}=e_{\coker(\alpha)}.
\]
It is also possible to check that the identities 
\[
(\alpha\beta)^{+}=(\alpha\beta^{+})^{+},\quad(\alpha\beta)^{\ast}=(\alpha^{\ast}\beta)^{\ast}
\]
are satisfied for every $\alpha,\beta\in\PR_{n}$. Therefore, $\PR_{n}$
is an Ehresmann semigroup and there exists a corresponding Ehresmann
category $\mathcal{C}_{n}$. This fact was first noted by East and
Gray \cite{East2020}. In this specific case it will be more convenient
to compose morphisms in a category ``from left to right'' because
it fits better the standard description of the partition monoid. The
objects of $\mathcal{C}_{n}$ are in one-to-one correspondence with
equivalence relations (or partitions) on ${\bf n}$ and for every
$\alpha\in\PR_{n}$ there is a corresponding morphism $C(\alpha)$
with domain $\ker(\alpha)$ and range $\coker(\alpha)$. Note that
$\mathcal{C}_{n}$ is not an EI-category. According to \cite{Halverson2005},
the algebra $\mathbb{C}\PR_{n}$ is not semisimple (for $n>1$). In
order to complete our counterexample it is enough to show that $\mathbb{C}\mathcal{C}_{n}$
is semisimple, hence $\Bbbk\PR_{n}\not\simeq\Bbbk\mathcal{C}_{n}$
for $\Bbbk=\mathbb{C}$. Denote by ${\bf 1}_{n}$ the identity relation
on ${\bf n}$. Note that $e_{{\bf 1}_{n}}$is the identity of $\PR_{n}$
which will also be denoted by ${\bf 1}_{n}$ for the sake of simplicity.
\begin{prop}
\label{prop:Isomorphisms_for_morita}Let $\Bbbk$ be a field.
\begin{enumerate}
\item The following equality holds
\[
\mathbb{\Bbbk}\mathcal{C}_{n}{\bf 1}_{n}\mathbb{\Bbbk}\mathcal{C}_{n}=\mathbb{\Bbbk}\mathcal{C}_{n}.
\]
\item There is an isomorphism of $\mathbb{\Bbbk}$-algebras
\[
{\bf 1}_{n}\mathbb{\Bbbk}\mathcal{C}_{n}{\bf 1}_{n}\simeq\mathbb{\Bbbk}\IS_{n}
\]
 (where $\IS_{n}$ is the symmetric inverse monoid).
\end{enumerate}
\end{prop}
\begin{proof}
\begin{enumerate}
\item Let $\alpha\in\PR_{n}$. Denote by $K_{1},\ldots,K_{r}$ the kernel
classes of the domain and by $K_{r+1},\ldots,K_{s}$ the other kernel
classes. Likewise, denote by $K_{1}^{\prime},\ldots,K_{r}^{\prime}$
the cokernel classes of the codomain and by $K_{r+1}^{\prime},\ldots,K_{s^{\prime}}^{\prime}$
the other cokernel classes ($s\neq s^{\prime}$ in general). We order
the classes such that $i\sim_{\alpha}j$ for $i\in K_{l}$ and $j\in K_{l}^{\prime}$
(for $l\leq r$). Define a partition $\alpha_{1}\in\PR_{n}$ by 
\[
\alpha_{1}=\{K_{1}\cup\{1^{\prime}\},\ldots,K_{r}\cup\{r^{\prime}\},K_{r+1},\ldots K_{s},\{(r+1)^{\prime}\},\ldots,\{n^{\prime}\}\}
\]
and another partition $\alpha_{2}\in\PR_{n}$ by
\[
\alpha_{2}=\{\{1\}\cup K_{1}^{\prime},\ldots,\{r\}\cup K_{r}^{\prime},\{r+1\},\ldots,\{n\},K_{r+1}^{\prime},\ldots,K_{s^{\prime}}^{\prime}\}.
\]
First note that $\alpha=\alpha_{1}\alpha_{2}=\alpha_{1}{\bf 1}_{n}\alpha_{2}$
and moreover, 
\[
\ker(\alpha_{1})=\ker(\alpha)
\]
\[
\coker(\alpha_{1})={\bf 1}_{n}=\ker(\alpha_{2})
\]
\[
\coker(\alpha_{2})=\coker(\alpha).
\]
This implies that 
\[
C(\alpha)=C(\alpha_{1})C({\bf 1}_{n})C(\alpha_{2})
\]
since the composition of morphisms on the right hand side is defined.
Therefore, $\Bbbk\mathcal{C}_{n}{\bf 1}_{n}\mathbb{\Bbbk}\mathcal{C}_{n}=\mathbb{\Bbbk}\mathcal{C}_{n}$
as required.
\item It is clear that $C(\alpha)\in{\bf 1}_{n}\mathcal{C}_{n}{\bf 1}_{n}$
if and only if $\ker(\alpha)=\coker(\alpha)={\bf 1}_{n}$ which implies
that $\alpha\in\IS_{n}$ (according to the natural embedding of $\IS_{n}$
in $\PR_{n}$, see \cite[Section 3]{East2011}). The claim follows
immediately.
\end{enumerate}
\end{proof}
\begin{cor}
\label{cor:Semisimple}$\mathbb{\Bbbk}\mathcal{C}_{n}$ is a semisimple
algebra if and only if $n!$ is invertible in $\Bbbk$.
\end{cor}
\begin{proof}
It is well known (see \cite[Theorem 2.8.7]{Linckelmann2018}) that
for any algebra $A$ and an idempotent $e\in A$, if $A=AeA$ then
the algebras $A$ and $eAe$ are Morita equivalent. In our case, we
take $A=\mathbb{C}\mathcal{C}_{n}$ and $e={\bf 1}_{n}$ so \propref{Isomorphisms_for_morita}
implies that $\mathbb{\Bbbk}\mathcal{C}_{n}$ is Morita equivalent
to $\mathbb{\Bbbk}\IS_{n}$. It is well known that $\mathbb{\Bbbk}\IS_{n}$
is a semisimple algebra if and only if $n!$ is invertible in $\Bbbk$
(\cite[Corollary 9.7]{Steinberg2016}) hence the same holds for $\mathbb{\Bbbk}\mathcal{C}_{n}$
as well. 
\end{proof}
\corref{Semisimple} for $\Bbbk=\mathbb{C}$ settles the counterexample
which is the main goal of this section.

Finally, we would like to give an example of an EI-Ehresmann semigroup
that is not left nor right restriction, but still has an algebra isomorphic
to the algebra of its category. Consider \exaref{EI_Ehresmann_not_Right_left_restriction}
given above. The category $\mathcal{C}$ is a poset. That is, for
any two objects $a,b$ there is at most one morphism in the union
of hom-sets $\mathcal{C}(a,b)\cup\mathcal{C}(b,a)$. Therefore, its
algebra is the incidence algebra of this poset. In this case the poset
is the linear poset $(\id,{\bf 1})<({\bf 1,1})<({\bf 1},\id)$ so
in this case the incidence algebra (over a field $\Bbbk$) is the
algebra $\UT_{3}(\Bbbk)$ of all upper triangular matrices over $\Bbbk$.
Now, consider the map $\mbox{\ensuremath{\Phi:S\to\UT_{3}(\Bbbk)}}$
defined by
\begin{align*}
({\bf 1},\id) & \to\left(\begin{array}{ccc}
1 & 0 & 0\\
0 & 1 & 0\\
0 & 0 & 0
\end{array}\right),\quad(\id,{\bf 1})\to\left(\begin{array}{ccc}
0 & 0 & 0\\
0 & 1 & 0\\
0 & 0 & 1
\end{array}\right)\\
({\bf 1,1}) & \to\left(\begin{array}{ccc}
0 & 0 & 0\\
0 & 1 & 0\\
0 & 0 & 0
\end{array}\right),\quad({\bf 2,1})\to\left(\begin{array}{ccc}
0 & 0 & 0\\
0 & 1 & 1\\
0 & 0 & 0
\end{array}\right)\\
({\bf 1,2}) & \to\left(\begin{array}{ccc}
0 & 1 & 0\\
0 & 1 & 0\\
0 & 0 & 0
\end{array}\right),\quad({\bf 2,2})\to\left(\begin{array}{ccc}
0 & 1 & 1\\
0 & 1 & 1\\
0 & 0 & 0
\end{array}\right).
\end{align*}
It is not difficult to check that $\Phi$ is a one-to-one semigroup
homomorphism and that its image is a basis of $\UT_{3}(\Bbbk)$. Therefore
$\Phi$ extends to an isomorphism $\mbox{\ensuremath{\Phi:\Bbbk S\to\UT_{3}(\Bbbk)}}$.
This implies that $\Bbbk S\simeq\Bbbk\mathcal{C}$ in this case. 

We leave as an open problem to determine necessary and sufficient
conditions for an EI-Ehresmann finite semigroup to have an algebra
isomorphic to that of its associated category algebra. It may be true
that every finite EI-Ehresmann semigroup has an algebra isomorphic
to the algebra of its category.

\textbf{Acknowledgments:} The authors thank the referee for his$\textbackslash$her
helpful comments. We thank Professor Michael Kinyon and Professor
Victoria Gould for a number of useful discussions and for the specific
results we noted in the body of the paper. We thank Professor Timothy
Stokes for pointing our attention to references \cite{Stokes2009}
and \cite{shain1970} and for discussions related to their content
and relation to this paper.

\bibliographystyle{plain}
\bibliography{library}

\begin{thebibliography}{10}

\bibitem{Almeida1994}
Jorge Almeida.
\newblock {\em Finite semigroups and universal algebra}, volume~3 of {\em
  Series in Algebra}.
\newblock World Scientific Publishing Co. Inc., River Edge, NJ, 1994.
\newblock Translated from the 1992 Portuguese original and revised by the
  author.

\bibitem{Assem2006}
Ibrahim Assem, Daniel Simson, and Andrzej Skowro{\'n}ski.
\newblock {\em Elements of the representation theory of associative algebras.
  {V}ol. 1}, volume~65 of {\em London Mathematical Society Student Texts}.
\newblock Cambridge University Press, Cambridge, 2006.
\newblock Techniques of representation theory.

\bibitem{universalalgebra}
Stanley Burris and H.~P. Sankappanavar.
\newblock {\em A course in universal algebra}, volume~78 of {\em Graduate Texts
  in Mathematics}.
\newblock Springer-Verlag, New York, 1981.

\bibitem{Curtis1966}
Charles~W Curtis and Irving Reiner.
\newblock {\em Representation theory of finite groups and associative
  algebras}, volume 356.
\newblock American Mathematical Soc., 1966.

\bibitem{East2011}
James East.
\newblock On the singular part of the partition monoid.
\newblock {\em Internat. J. Algebra Comput.}, 21(1-2):147--178, 2011.

\bibitem{East2020}
James East and Robert~D Gray.
\newblock Ehresmann theory and partition monoids.
\newblock {\em arXiv preprint arXiv:2011.00663}, 2020.

\bibitem{Ganyushkin2009b}
Olexandr Ganyushkin and Volodymyr Mazorchuk.
\newblock {\em Classical finite transformation semigroups}, volume~9 of {\em
  Algebra and Applications}.
\newblock Springer-Verlag London Ltd., London, 2009.
\newblock An introduction.

\bibitem{Ganyushkin2009a}
Olexandr Ganyushkin, Volodymyr Mazorchuk, and Benjamin Steinberg.
\newblock On the irreducible representations of a finite semigroup.
\newblock {\em Proc. Amer. Math. Soc.}, 137(11):3585--3592, 2009.

\bibitem{Gould2010}
Victoria Gould.
\newblock Notes on restriction semigroups and related structures; formerly
  (weakly) left {E}-ample semigroups, 2010.

\bibitem{Gould2010b}
Victoria Gould.
\newblock Restriction and {E}hresmann semigroups.
\newblock In {\em Proceedings of the {I}nternational {C}onference on {A}lgebra
  2010}, pages 265--288. World Sci. Publ., Hackensack, NJ, 2012.

\bibitem{Guo2008}
Xiaojiang Guo.
\newblock A note on locally inverse semigroup algebras.
\newblock {\em Int. J. Math. Math. Sci.}, 2008.
\newblock Art. ID 576061, 5.

\bibitem{Guo2012}
Xiaojiang Guo and Lin Chen.
\newblock Semigroup algebras of finite ample semigroups.
\newblock {\em Proc. Roy. Soc. Edinburgh Sect. A}, 142(2):371--389, 2012.

\bibitem{Halverson2005}
Tom Halverson and Arun Ram.
\newblock Partition algebras.
\newblock {\em European J. Combin.}, 26(6):869--921, 2005.

\bibitem{Howie1995}
John~M. Howie.
\newblock {\em Fundamentals of semigroup theory}, volume~12 of {\em London
  Mathematical Society Monographs. New Series}.
\newblock The Clarendon Press Oxford University Press, New York, 1995.
\newblock Oxford Science Publications.

\bibitem{Stokes2009}
Marcel Jackson and Tim Stokes.
\newblock Partial maps with domain and range: extending {S}chein's
  representation.
\newblock {\em Comm. Algebra}, 37(8):2845--2870, 2009.

\bibitem{James1981}
Gordon James and Adalbert Kerber.
\newblock {\em The representation theory of the symmetric group}, volume~16 of
  {\em Encyclopedia of Mathematics and its Applications}.
\newblock Addison-Wesley Publishing Co., Reading, Mass., 1981.
\newblock With a foreword by P. M. Cohn, With an introduction by Gilbert de B.
  Robinson.

\bibitem{Ji2016}
Yingdan Ji and Yanfeng Luo.
\newblock Locally adequate semigroup algebras.
\newblock {\em Open Math.}, 14:29--48, 2016.

\bibitem{Jones2012}
Peter~R. Jones.
\newblock A common framework for restriction semigroups and regular
  {$*$}-semigroups.
\newblock {\em J. Pure Appl. Algebra}, 216(3):618--632, 2012.

\bibitem{Lawson1991}
M.~V. Lawson.
\newblock Semigroups and ordered categories. {I}. {T}he reduced case.
\newblock {\em J. Algebra}, 141(2):422--462, 1991.

\bibitem{Lawson1998}
Mark~V. Lawson.
\newblock {\em Inverse semigroups}.
\newblock World Scientific Publishing Co., Inc., River Edge, NJ, 1998.
\newblock The theory of partial symmetries.

\bibitem{Li2011}
Liping Li.
\newblock A characterization of finite {EI} categories with hereditary category
  algebras.
\newblock {\em J. Algebra}, 345:213--241, 2011.

\bibitem{Linckelmann2018}
Markus Linckelmann.
\newblock {\em The block theory of finite group algebras. {V}ol. {I}},
  volume~91 of {\em London Mathematical Society Student Texts}.
\newblock Cambridge University Press, Cambridge, 2018.

\bibitem{Luck1989}
Wolfgang L\"{u}ck.
\newblock {\em Transformation groups and algebraic {$K$}-theory}, volume 1408
  of {\em Lecture Notes in Mathematics}.
\newblock Springer-Verlag, Berlin, 1989.
\newblock Mathematica Gottingensis.

\bibitem{Margolis2011}
Stuart Margolis and Benjamin Steinberg.
\newblock The quiver of an algebra associated to the {M}antaci-{R}eutenauer
  descent algebra and the homology of regular semigroups.
\newblock {\em Algebr. Represent. Theory}, 14(1):131--159, 2011.

\bibitem{Margolis2012}
Stuart Margolis and Benjamin Steinberg.
\newblock Quivers of monoids with basic algebras.
\newblock {\em Compos. Math.}, 148(5):1516--1560, 2012.

\bibitem{Margolis2018B}
Stuart Margolis and Benjamin Steinberg.
\newblock Projective indecomposable modules and quivers for monoid algebras.
\newblock {\em Comm. Algebra}, 46(12):5116--5135, 2018.

\bibitem{Prover9}
W.~McCune.
\newblock Prover9 and mace4.
\newblock \verb|http://www.cs.unm.edu/~mccune/prover9/|, 2005--2010.

\bibitem{qTheory2009}
John Rhodes and Benjamin Steinberg.
\newblock {\em The {$q$}-theory of finite semigroups}.
\newblock Springer Monographs in Mathematics. Springer, New York, 2009.

\bibitem{Sagan2001}
Bruce~E. Sagan.
\newblock {\em The symmetric group}, volume 203 of {\em Graduate Texts in
  Mathematics}.
\newblock Springer-Verlag, New York, second edition, 2001.
\newblock Representations, combinatorial algorithms, and symmetric functions.

\bibitem{shain1970}
Boris~M Shain.
\newblock Restrictively multiplicative algebras of transformations.
\newblock {\em Izvestiya Vysshikh Uchebnykh Zavedenii. Matematika},
  (4):91--102, 1970.

\bibitem{Solomon1967}
Louis Solomon.
\newblock The {B}urnside algebra of a finite group.
\newblock {\em J. Combinatorial Theory}, 2:603--615, 1967.

\bibitem{Stanley1997}
Richard~P. Stanley.
\newblock {\em Enumerative combinatorics. {V}ol. 1}, volume~49 of {\em
  Cambridge Studies in Advanced Mathematics}.
\newblock Cambridge University Press, Cambridge, 1997.
\newblock With a foreword by Gian-Carlo Rota, Corrected reprint of the 1986
  original.

\bibitem{Stein2016}
Itamar Stein.
\newblock The representation theory of the monoid of all partial functions on a
  set and related monoids as {EI}-category algebras.
\newblock {\em J. Algebra}, 450:549--569, 2016.

\bibitem{Stein2017}
Itamar Stein.
\newblock Algebras of {E}hresmann semigroups and categories.
\newblock {\em Semigroup Forum}, 95(3):509--526, 2017.

\bibitem{Stein2018erratum}
Itamar Stein.
\newblock Erratum to: {A}lgebras of {E}hresmann semigroups and categories.
\newblock {\em Semigroup Forum}, 96(3):603--607, 2018.

\bibitem{Stein2019}
Itamar Stein.
\newblock The global dimension of the algebra of the monoid of all partial
  functions on an {$n$}-set as the algebra of the {EI}-category of epimorphisms
  between subsets.
\newblock {\em J. Pure Appl. Algebra}, 223(8):3515--3536, 2019.

\bibitem{Stein2020}
Itamar Stein.
\newblock Representation theory of order-related monoids of partial functions
  as locally trivial category algebras.
\newblock {\em Algebr. Represent. Theory}, 23(4):1543--1567, 2020.

\bibitem{Steinberg2006}
Benjamin Steinberg.
\newblock M\"obius functions and semigroup representation theory.
\newblock {\em J. Combin. Theory Ser. A}, 113(5):866--881, 2006.

\bibitem{Steinberg2008}
Benjamin Steinberg.
\newblock M\"obius functions and semigroup representation theory. {II}.
  {C}haracter formulas and multiplicities.
\newblock {\em Adv. Math.}, 217(4):1521--1557, 2008.

\bibitem{Steinberg2016}
Benjamin Steinberg.
\newblock {\em Representation theory of finite monoids}.
\newblock Universitext. Springer, Cham, 2016.

\bibitem{Dieck1987}
Tammo tom Dieck.
\newblock {\em Transformation groups}, volume~8 of {\em De Gruyter Studies in
  Mathematics}.
\newblock Walter de Gruyter \& Co., Berlin, 1987.

\bibitem{Wang2017}
Shoufeng Wang.
\newblock On algebras of {$P$}-{E}hresmann semigroups and their associate
  partial semigroups.
\newblock {\em Semigroup Forum}, 95(3):569--588, 2017.

\bibitem{Webb2007}
Peter Webb.
\newblock An introduction to the representations and cohomology of categories.
\newblock In {\em Group representation theory}, pages 149--173. EPFL Press,
  Lausanne, 2007.

\end{thebibliography}

\end{document}